\documentclass[reqno]{amsart}
\usepackage{amssymb,latexsym}
\usepackage{amsmath}
\usepackage{graphicx}
\usepackage{amscd}
\usepackage{color}
\usepackage{enumerate}

\numberwithin{equation}{section}
\theoremstyle{plain}
 \newtheorem{theorem}{Theorem}[section]
 \newtheorem{lemma}[theorem]{Lemma}
 \newtheorem{proposition}[theorem]{Proposition}

 \newtheorem{corollary}[theorem]{Corollary}
\theoremstyle{definition}
 
 \newtheorem{remark}[theorem]{Remark}
 \newtheorem{notation}[theorem]{Notation}
 \newtheorem{convention}[theorem]{Convention}
 \newtheorem{example}[theorem]{Example}
 \newtheorem{problem}[theorem]{Problem}
 
 \newtheorem{conjecture}[theorem]{Conjecture}
\theoremstyle{remark}
 \newtheorem{case}{Case}

\newcommand \Part[1] {\textup {Part}({#1})}
\newcommand \gnu [1] {\nu(#1)}
\newcommand \grho {p}

\newcommand \ogrho  {\overline p}
\newcommand \isitneeded[1]{}
\newcommand \ovl [1] {\overline #1}
\newcommand \conf {\alpha_\textup{conf}}
\newcommand \zconf {z(\alpha_\textup{conf})}
\newcommand \othv {r} 
\renewcommand \epsilon{\varepsilon}

\newcommand \upint [1] {\lceil #1 \rceil}
\newcommand \loint [1] {\lfloor #1 \rfloor}

\newcommand \zwid [1] {\textup{zwid}(#1)}
\newcommand \wid [1] {\textup{wid}(#1)}
\newcommand \vdim [1] {\textup{dim}(#1)}
\newcommand \zlen [1] {\textup{zlen}(#1)}
\newcommand \len [1] {\textup{len}(#1)}
\newcommand \air [1] {\textup{air}(#1)}
\newcommand  \eff [1] {{\textup{eff}(#1)}}

\newcommand \Sym [1] {\textup{Sym}(#1)}
\newcommand \Stb [1] {\textup{Stb}(#1)}
\newcommand \wextra  {w}
\newcommand \notdiv {\mathrel{\not{\kern-0.3pt|}}}

\newcommand  \Zfi {Z[\vph]}
\newcommand  \Zpfi {Z[\pvph]}
\newcommand  \Zmfi {Z[\mvph]}

\newcommand  \Zppfi[1] {Z[#1]}

\newcommand \inseg[2] {\textup{init}(#1,#2)}

\newcommand \stza {a^{\vph}}
\newcommand \stzb {b^{\vph}} 
\newcommand \stzc {c^{\vph}} 
\newcommand \stpza {a^{\pvph}} 
\newcommand \stpzb {b^{\pvph}} 
\newcommand \stpzc {c^{\pvph}} 
\newcommand \stzga {\alpha^{\vph}} 
\newcommand \stzgb {\beta^{\vph}}
\newcommand \stzgc {\gamma^{\vph}}
\newcommand \stzgd {\delta^{\vph}}

\newcommand \stpzga {\alpha^{\pvph}} 
\newcommand \stpzgb {\beta^{\pvph}}
\newcommand \stpzgc {\gamma^{\pvph}}
\newcommand \stpzgd {\delta^{\pvph}}
\newcommand \stpzge {\epsilon^{\pvph}}

\newcommand \stzogb {{\overline\beta}^{\kern 1pt\vph}} 
\newcommand \stpzogb {{\overline\beta}{}^{\kern 1pt\pvph}\kern-2pt } 
\newcommand \stzogc {{\overline\gamma}^{\kern 1pt\vph}} 
\newcommand \stpzogc {{\overline\gamma}{}^{\kern 1pt\pvph}\kern-2pt } 

\newcommand \stzogtwa {{\overline\alpha}_2^{\kern 1pt\vph}} 
\newcommand \stzogtwb {{\overline\beta}_2^{\kern 1pt\vph}} 
\newcommand \stzogtwc {{\overline\gamma}_2^{\kern 1pt\vph}} 
\newcommand \stzogtwd {{\overline\delta}_2^{\kern 1pt\vph}} 

\newcommand \stpzogtwa {{\overline\alpha}_2^{\kern 1pt\pvph}\kern-2pt} 
\newcommand \stpzogtwb {{\overline\beta}_2{}^{\kern -3.5pt{\raisebox{1pt}[0pt]{$\scriptstyle\pvph$}}\kern-2pt}} 
\newcommand \stpzogthb {{\overline\beta}_3{}^{\kern -3.5pt{\raisebox{1pt}[0pt]{$\scriptstyle\pvph$}}\kern-2pt}} 

\newcommand \stpzogtwc {{\overline\gamma}_2^{\kern 1pt\pvph}\kern-2pt} 
\newcommand \stpzogthc {{\overline\gamma}_3^{\kern 1pt\pvph}\kern-2pt}

\newcommand \stpzogthd {{\overline\delta}_3{}^{\kern -3pt{\raisebox{1.5pt}[0pt]{$\scriptstyle\pvph$}}\kern-2pt}}

\newcommand \stzogtwbmu {{\bmu}_2^{\kern 1pt\vph}} 
\newcommand \stpzogtwobmu {{\obmu}_2^{\kern 1pt\pvph}\kern-2pt}

\newcommand \stzf {f^{\vph}} 
\newcommand \stpzf {f^{\pvph}} 
\newcommand \stzbg {^\bullet\kern-2pt g^{\vph}} 
\newcommand \stzclubg {^\clubsuit\kern-2pt g^{\vph}} 
\newcommand \stpzbg {^\bullet\kern-2pt g^{\pvph}\kern-2pt } 

\newcommand \stzotbmu{{\overline{\widetilde{\pmb\mu}}_2{}^{\kern -3.5pt{\raisebox{1pt}[0pt]{$\scriptstyle\vph$}}\kern-0.5pt}}}

\newcommand \stpzotbmu{{\overline{\widetilde{\pmb\mu}}_2{}^{\kern -3.5pt{\raisebox{1pt}[0pt]{$\scriptstyle\pvph$}}\kern-2pt}}}

\newcommand \stpzclubg {^\clubsuit\kern-2pt g^{\pvph}} 

\newcommand \stmza {a^{\mvph}} 
\newcommand \stmzb {b^{\mvph}}

\newcommand \stzg {g^{\vph}} 
\newcommand \stpzg {g^{\pvph}} 
\newcommand \stzdotsg {\ddot g^{\kern0.8pt\vph}} 
\newcommand \stpzdotsg {\ddot g^{\kern0.8pt\pvph}} 
\newcommand \stzkg {\widehat g^{\kern1.8pt\vph}} 
\newcommand \stpzkg {\widehat g^{\kern1.8pt\pvph}} 

\newcommand \stzkf {\widehat f^{\kern1.8pt\vph}} 
\newcommand \stzbmu {{\bmu}^{\vph}} 
\newcommand \stzI {I^{\vph}} 
\newcommand \stzeuv [2] {e^{\vph}_{#1,#2}} 
\newcommand \stpzeuv [2] {e^{\pvph}_{#1,#2}} 
\newcommand \stpzclubeuv [2] {{}^\clubsuit\kern-2pt e^{\pvph}_{#1,#2}} 
\newcommand \stzclubeuv [2] {{}^\bullet\kern-2pt e^{\pvph\kern-2pt,\kern1pt\vph}_{#1,#2}} 
\newcommand \stzzclubeuv [2] {{}^\bullet\kern-2pt e^{\vph\kern0pt,\kern1pt\vph}_{#1,#2}}

\newcommand \stmzeuv [2] {e^{\mvph}_{#1,#2}} 
\newcommand \stzteuv [2] {\widetilde e^{\kern 1pt\vph}_{#1,#2}}

\newcommand \restrict[2] {{#1\rceil_{#2}}}
\newcommand \sublat[1] {[#1]_{\kern-1pt\textup{lat}}}

\newcommand \lbrak {[\hskip-1.5pt[}
\newcommand \rbrak {\hskip0.5pt]\hskip-1.5pt]}
\newcommand \equ[2]{\lbrak{}#1,#2\rbrak{}^{\kern-0.5pt{
\scriptscriptstyle\textup{e}}}}
\newcommand \kequ[1]{\lbrak{}#1\rbrak{}^{\kern-0.5pt{\scriptscriptstyle\textup{e}}}}
\newcommand \xequ[3]{\lbrak{}\langle #1,#2\rangle : #3\rbrak{}^{\kern-0.5pt{\scriptscriptstyle\textup{e}}}}
\newcommand \faequ[2]{\lbrak{}#1,#2\rbrak{}^{\kern-0.5pt{\scriptscriptstyle\textup{e}}}_{\kern-0.5pt\scriptscriptstyle\forall}}
\newcommand \vonal {\noalign{\hrule}}
\newcommand \tuple [1] {\langle #1 \rangle}
\newcommand \pair [2] {\tuple{#1,#2}}
\newcommand \oal{\overline\alpha}
\newcommand \obe{\overline\beta}
\newcommand \oga{\overline\gamma}
\newcommand \ode{\overline\delta}
\newcommand \valpha{\vec\alpha}
\newcommand \vbeta{\vec\beta}
\newcommand \vgamma{\vec\gamma}
\newcommand \vdelta{\vec\delta}
\newcommand \stxal[3]{{\alpha^{\pair{#1}{#2}}[#3\kern 0.8pt]}}
\newcommand \stxbe[3]{{\beta^{\pair{#1}{#2}}[#3\kern 0.8pt]}}
\newcommand \stxga[3]{{\gamma^{\pair{#1}{#2}}[#3\kern 0.8pt]}}
\newcommand \stxde[3]{{\delta^{\pair{#1}{#2}}[#3\kern 0.8pt]}}
\newcommand \stxn[3]{{n^{\pair{#1}{#2}}[#3\kern 0.8pt]}}
\newcommand \stxm[3]{{m^{\pair{#1}{#2}}[#3\kern 0.8pt]}}
\newcommand \stxk[3]{{k^{\pair{#1}{#2}}[#3\kern 0.8pt]}}
\newcommand \vbmu{{{\vec{\pmb\mu}\kern 0pt ^\ast}}}
\newcommand \DBV[2]{\textup{DBV}(#1,#2) }
\newcommand \sba[2]{\textup{sba}(#1,#2) }
\newcommand \TRA[2]{\textup{TRA}(#1,#2) }
\newcommand \tra[2]{\textup{tra}(#1,#2) }

\newcommand \obmu{{\overline{\pmb\mu}}}

\newcommand \bomu{{^\ast\kern0pt\overline{\pmb\mu}}}
\newcommand \bmu{{{\pmb\mu}}}

\newcommand \cxbmu [3]{{{\pmb\mu}}^{\pair{#1}{#2}}[#3\kern0.8pt]}

\newcommand \rbmu{{^\ast\kern-1.5pt{\pmb\mu}}}

\newcommand \lev {\kern 1pt|}

\newcommand \hkey[2] {h_{#1}[#2\kern 1pt]}
\newcommand \ehkey[4] {h^{\pair{#1}{#2}}_{#3}[#4\kern 1pt]}

\newcommand \rftrm[1] {{^\ast\kern-2pt f_{#1}}}

\newcommand \sltrm[1]  {g^{\textup{left}}}
\newcommand \srtrm[1]  {g^{\textup{right}}}
\newcommand \keterm [2] {\widehat e_{#1,#2}}

\newcommand \kecsterm [2]{{\widehat e_{#1,#2}}^{\kern 2pt\clubsuit}}
\newcommand \kgcsterm [2]{{\widehat g_{#1,#2}}^{\kern 2pt\clubsuit}}
\newcommand \kfcsterm [2]{{f^{\kern 0pt\clubsuit\clubsuit}_{#1,#2}}}

\newcommand \enul {\Delta}

\newcommand \NN {{\mathbb N^+}}
\newcommand \Nnul {{\mathbb N_0}}

\newcommand \Bell[1] {\textup{Bell}(#1)}

\newcommand \leqref [1] {\overset{\eqref{#1}}{\leq}}
\newcommand \eeqref[1]{\overset{\eqref{#1}}{=}}

\newcommand \pvec[1] {{\vec{#1}\,\pmb'}}

\newcommand \ppvec[1] {{\vec{#1}\kern 1pt''}}

\newcommand \ffree[1] {\textup{FL}(3)}

\newcommand \lideal[1]{\mathord{\downarrow}_{\kern-1pt#1}}
\newcommand \lfilter[1]{\mathord{\uparrow}_{\kern-1pt#1}}

\renewcommand\phi{\varphi}
\newcommand \vph{{\boldsymbol\phi}}
\newcommand \pvph{{\boldsymbol{\phi'}}}
\newcommand \mvph{{\boldsymbol{\phi^{\kern-1pt\mathord-}}}}

\newcommand \tbf [1] {\textbf{#1}} 
\newcommand \set[1] {\{#1\}}
\newcommand \Equ[1] {\textup{Equ}(#1)}

\renewcommand \phi {\varphi}

\newcommand \red [1] {{\color{red}#1\color{black}}}

\newcommand \nothing [1] {}
\newcommand \magenta [1] {{\color{magenta}#1\color{black}}}

\begin{document}
\title[Generating partition lattices and their direct products] 
{Four-element generating sets of partition lattices and their direct products}

%
\author[G.\ Cz\'edli]{G\'abor Cz\'edli}
\email{czedli@math.u-szeged.hu}
\urladdr{http://www.math.u-szeged.hu/\textasciitilde{}czedli/}
\address{University of Szeged, Bolyai Institute, 
Szeged, Aradi v\'ertan\'uk tere 1, HUNGARY 6720}

\author[L.\ Oluoch]{Lillian Oluoch}
\email{lillybenns@yahoo.com}
\urladdr{http://www.math.u-szeged.hu/\textasciitilde{}oluoch/}
\address{University of Szeged, Bolyai Institute, 
Szeged, Aradi v\'ertan\'uk tere 1, HUNGARY 6720}

\thanks{This research of the first author is supported by  NFSR of Hungary (OTKA), grant
number K 134851}

\dedicatory{Dedicated to the memory of our colleague, Professor Gyula Pap $(\kern-1pt$1954--2019$\kern1pt)$
}

\date{\hfill {\tiny{\magenta{(\tbf{\underline{Always}} check the  authors' websites for possible updates!) }}}\  \red{June 24, 2020}}



\begin{abstract} 
Let $n>3$ be a natural number.
By a 1975 result of H. Strietz,  the lattice  Part$(n)$ of all partitions of an $n$-element set  has a four-element generating set.
In 1983, L.\ Z\'adori gave a new proof of this fact with a particularly elegant construction. Based on his construction from 1983, the present paper  gives a lower bound on the number $\gnu n$ of four-element  generating sets of Part$(n)$. 
We also present a computer assisted statistical approach to $\gnu n$ for small values of $n$. 

In his 1983 paper, L.\ Z\'adori also proved that for $n\geq 7$,
the lattice Part$(n)$ has a four element generating set that is not an antichain. He left the problem whether such a  generating set for $n\in\set{5,6}$ exists open. Here we solve this problem in negative for $n=5$ and in affirmative for $n=6$. 

Finally, the main theorem asserts that the direct product of some powers of partition lattices is four-generated. In particular, by the first part of this theorem,  $\Part{n_1}\times\Part{n_2}$ is four-generated for any two distinct integers $n_1$ and $n_2$ that are at least 5. The second part of the theorem is technical but it has two corollaries that are easy to understand. Namely, the direct product $\Part{n}\times \Part{n+1}\times\dots\times \Part {3n-14}$ is four-generated for each integer $n\geq 9$. Also, for every positive integer $u$, the $u$-th the direct power of the direct product 
$\Part{n}\times \Part{n+1}\times\dots\times \Part{n+u-1}$ is four-generated for all but finitely many $n$. If we do not insist on too many direct factors, then the exponent can be quite large. For example, our  theorem implies that the $ 10^{127}$-th direct power of $\Part{1011}\times \Part{1012}\times \dots \times\Part{2020}$ is four-generated. 
\end{abstract}

\maketitle
\section{Introduction}\label{sectintro}
\subsection*{Goal}
For a non-empty set $A$, let $\Part A$ denote the lattice of all partitions of $A$. Also, for a positive integer $n$, let $\Part n$ stand for  $\Part{\set{1,2,\dots,n}}$. Note that $\Part n$ is a semimodular and, in fact, a geometric lattice, which occurs frequently in lattice theory, universal algebra, and combinatorics.
We know from Strietz~\cite{strietz75,strietz77} that for $n\geq 3$, $\Part n$ is \emph{four-generated}, that is, it has a four-element generating set. 
We also know that $\Part n$ is not three-generated provided $n\geq 4$. A four-element generating set, different from and, in fact, more elegant than Strietz's, was found by Z\'adori~\cite{zadori}. Implicitly, some other four-element generating sets occur in 
Cz\'edli~\cite{czedli4gen,czsmall,czedli112gen}, which deal with infinite partition lattices. Hence, taking into account that $\Part n$ has $n!$ many automorphisms (those induced by the permutations of $\set{1,\dots,n}$) and that automorphisms send generating sets to generating sets, it is clear that $\Part n$ has \emph{many} four-element generating sets. It is natural to ask how many. So,  for $n\in\NN:=\set{1,2,3,\dots}$, 
\begin{equation}\left.
\parbox{6.5cm}{let us denote the number of four-element generating sets of $\Part n$ by $\gnu n$;}\,\,\right\}
\label{pbxgnujl}
\end{equation}
our first goal is to investigate this number. In addition to the fact that it is generally reasonable to enumerate our interesting objects in lattice theory, another ingredient of our motivation is given in an earlier and somewhat parallel   paper, Cz\'edli~\cite[last section]{czgparallel}. 
When dealing with $\gnu n$, it causes some difficulty that  $\Part n$ has very many elements and a complicated structure; for example, each finite lattice is embeddable in $\Part n$ for some $n$ by a classical result of Pudl\'ak and  T\r uma~\cite{pudltum}. This explains that, instead of determining the exact or the asymptotic value of $\gnu n$, we can only give a lower bound for $\gnu n$. This lower bound is better than what would trivially follow from Z\'adori~\cite{zadori} and Cz\'edli~\cite{czedli4gen,czsmall,czedli112gen} with $n!$ automorphisms taken into account; see above, before \eqref{pbxgnujl}. We prove the validity of our lower bound by constructing involved four-element generating sets of $\Part n$; these sets are  motivated by but differ from those occurring in the afore-mentioned papers. 

Since the lower bound for $\gnu n$ that we are able to prove is not sharp, we have developed some computer programs for investigating $\gnu n$ for some small values of $n$. The results obtained in this way are analyzed with the usual methods of mathematical statistics.

As an immediate use of our efforts to estimate $\gnu n$, we have found and proved our main result, Theorem~\ref{thmmain}, which states that certain direct products of direct powers of partitions lattices are still four-generated. For example, for any integers $5\leq n_1<n_2$, the direct product $\Part{n_1}\times \Part{n_2}$ is four-generated. If the reader decides to jump ahead to cast a quick glance at the technical second part of the main result,  he is advised to look at Corollaries~\ref{coroltconsecutive} and \ref{corolmanyfactors} first. Note that the direct products in the main result are geometric lattices; indeed, they are semimodular for example by Cz\'edli and Walendziak~\cite{czedliwalendziak}, and they are clearly atomistic.

\subsection*{Outline}
In Section~\ref{section112}, we recall some basic facts and definitions,  and we fix some notation and convention used in the whole paper. Also, Propositions~\ref{propsTsx} and \ref{propfVjh} solve Z\'adori's problem about the existence of such a four-element generating set of $\Part 5$ and $\Part 6$ that is not an antichain. For 
 $5\leq n\in\NN:=\set{1,2,3,\dots}$, Section~\ref{sectlowbnd} gives a lower bound of the  number  $\gnu n$ of four-element generating sets of $\Part n$. 
Section~\ref{sectdirprod} formulates and  proves Theorem~\ref{thmmain} about the existence of four-element generating sets in certain direct products of direct powers of partitions lattices; Corollaries~\ref{coroltconsecutive} and \ref{corolmanyfactors} as well as Example~\ref{example2020} are worth separate mentioning here. Finally, Section~\ref{sectstatcomp} outlines what our computer programs do, gives some facts achieved by these programs, and draws some conclusions and conjectures supported by mathematical statistics.

\subsection*{Joint authorship} It is not common that a lattice theorist and an expert of mathematical statistics work jointly. This unusual circumstance explains that we give some details of the authors' contributions. The two purely lattice theoretical sections, Sections~\ref{sectlowbnd} and \ref{sectdirprod}, are due to the first author. Sections~\ref{section112} and \ref{sectstatcomp}, that is, both sections that benefit from computers, are joint work, and so are most of the data obtained by computers. The statistical analysis of the data in Section  \ref{sectstatcomp} is due to the second author.

Let us take the opportunity to draw a general conclusion of this sort of joint work. Sometimes, research in lattice theory as well as in some other fields of mathematics produces numerical data. Even if statistics is not a substitute for a mathematical proof, it can give confidence in conjectures and in directions we try to go further.

\section{$(1+1+2)$-generation}\label{section112}
We know from  Strietz~\cite{strietz75,strietz77} and  Z\'adori~\cite{zadori} that, for $n\geq 4$, if $\Part n$ has a four-element generating set, then this generating set is a so-called $(1+1+2)$-poset, that is, exactly two of the four generators are comparable. If $\Part n$ has such a generating set, then we say that $\Part n$ has a \emph{$(1+1+2)$-generating set}.
The purpose of this section is to prove the following two statements, which solve a problem raised by Z\'adori~\cite{zadori}

\begin{proposition}\label{propsTsx}
The partition lattice $\Part 6$ has a $(1+1+2)$-generating set.
\end{proposition}

\begin{proposition}\label{propfVjh}
Every four-element generating set of $\Part 5$ is an antichain. Hence,  $\Part 5$ has no $(1+1+2)$-generating set.
\end{proposition}

\begin{figure}[ht]
\centerline
{\includegraphics[scale=0.9]{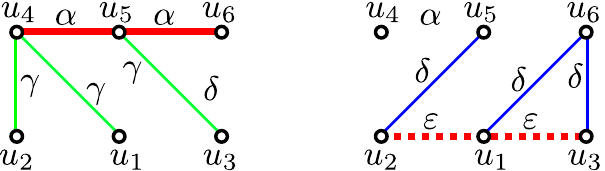}}
\caption{A $(1+1+2)$-generating set for $\Equ 6$}
\label{figequ6112}
\end{figure}

As usual, associated with a partition $U$ of $A$, we define an \emph{equivalence relation} $\pi_U$ of $A$ as the collection of all pairs $(x,y)\in A^2$ such that $x$ and $y$ belong to the same block of $U$. As it is well known, the equivalence relations and the partitions of $A$ mutually determine each other, and $\pi_U\leq \pi_V$ if and only if $U\leq V$. 
(Here $\pi_U\leq \pi_V$ means that $\pi_U\subseteq \pi_V$ as sets of pairs of elements of $A$.)
Hence, the \emph{lattice $\Equ A$ of all equivalence relations} of $A$ (in short, the \emph{equivalence lattice} of $A$) is isomorphic to $\Part A$. In what follows, we do not make a sharp distinction between a partition and the corresponding equivalence relation; no matter which of them is given, we can use the other one without warning.
Typically, we speak of \emph{partitions lattices} in the main statements but we prefer to speak of \emph{equivalence lattices} in the proofs.

\begin{convention}\label{conventionsgRsktWph}
Almost exclusively, we are going to define our equivalence relations and the corresponding partitions by (undirected edge-coloured)  graphs; multiple edges are allowed. Even when we give an equivalence relation by a defining equation, there is a graph in the background. 
 Each edge of the graph is colored by one of the colors $\alpha$, $\beta$, $\gamma$,  $\delta$, and $\epsilon$. 
On the vertex set of such a graph $A$, Figure~\ref{figequ6112} and the other figures in the paper define an \emph{equivalence} (relation) $\alpha\in \Equ A$ in the following way: deleting all edges but the $\alpha$-colored ones, the components of the remaining graph are the blocks of the partition associated with $\alpha$. In other words, $\pair x y\in\alpha$ if and only if there is an $\alpha$-coloured path from vertex $x$ to vertex $y$ in the graph, that is, a path (of possibly zero length) all of whose edges are $\alpha$-colored. The equivalences $\beta$, $\gamma$, $\delta$, and (sometimes) $\epsilon$ are defined analogously.
\end{convention}

\begin{notation} Following Cz\'edli~\cite{czgparallel}, we adopt the following notation. Assume that $A$ is a base set and we are interested in its partitions or, equivalently, in its equivalence relations.
For elements $u_1,\dots,u_k$ of $A$, the partition of $A$ with block $\set{u_1,\dots, u_k}$ such that all the other blocks are singletons will be denoted by
\[\kequ{u_1,\dots u_k}.
\] 
Usually but not always, the elements $u_1,\dots, u_k$ are assumed to be pairwise distinct. Note that $\kequ{u_1,u_1}$ is $\enul$, the least equivalence relation of $A$, that is, the zero element of $\Equ A$. 
For $\kappa,\lambda\in \Equ A$, the \emph{meet} and the \emph{join} of $\kappa$ and $\lambda$, denoted by $\kappa \lambda$ (or $\kappa\cdot\lambda$) and $\kappa+\lambda$, are the intersection and the transitive hull of the union of $\kappa$ and $\lambda$, respectively.  The usual precedence rules apply; for example, $xy+xz$ stands for $(x\wedge y)\vee (x\wedge z)$. 
\emph{Lattice terms} are composed from variables and join and meet operation signs in the usual way; for example, $f(x_1,x_2,x_3,x_4)=x_1(x_3+x_4)+(x_1+x_3)x_4$ is a quaternary lattice term. 
Given a lattice $L$ and $u_1,\dots, u_k\in L$, the \emph{sublattice generated} by $\set{u_1,\dots,u_k}$ is denoted and defined by 
\begin{equation}
\sublat{u_1,\dots,u_k}:=\set{f(u_1,\dots,u_k): u_1,\dots,u_k\in L,\,\,f\text{ is a lattice term}}.
\end{equation}
\end{notation}

Our arguments will often use the following technical lemma from Z\'adori \cite{zadori}, which has been used also in Cz\'edli \cite{czedli4gen,czsmall,czedli112gen} and in some other papers like Kulin~\cite{kulin}. Note that the proof of this lemma is straightforward.

\begin{lemma}[``Circle Principle'']\label{lemmaHamilt}
Let  $d_0,d_1,\dots,d_{t-1}$ be pairwise distinct elements of a set $A$. 
Then, for any $0\leq i<j\leq t-1$ and in the lattice $\Equ A$, 
\begin{equation}
\begin{aligned}
\equ {d_i}{d_j}=\bigl(\equ{d_{i}}{d_{i+1}} + \equ{d_{i+1}}{d_{i+2}}\dots + \equ{d_{j-1}}{d_{j}} \bigr) \cdot
 \bigl( \equ{d_{j}}{d_{j+1}} 
\cr
+ \dots + \equ{d_{t-2}}{d_{t-1}} +
\equ{d_{t-1}}{d_{0}}
 + \equ{d_{0}}{d_{1}}+ \dots + \equ{d_{i-1}}{d_{i}} \bigr).
\end{aligned}
\label{eqGbVrTslcdNssm}
\end{equation}
Consequently, $\equ {d_i}{d_j}\in
\sublat{\, \equ {d_0}{d_1}, \equ{d_1}{d_2},\dots, \equ{d_{t-2}}{d_{t-1}}, \equ{d_{t-1}}{d_0}\, }$.
\end{lemma}

For later reference, note the following. If all the joinands (formally, the summands) in \eqref{eqGbVrTslcdNssm} are 
substitution values of appropriate quaternary terms, then so is $\equ {d_i}{d_j}$ of a longer quaternary term, which is defined according to \eqref{eqGbVrTslcdNssm} and 
\begin{equation}
\text{which we denote by $\keterm {d_i}{d_j}$.}
\label{eqtxthszKthTdtrDpcRvszQ}
\end{equation}

Armed with our conventions and notations, the first proof in the paper runs as follows.

\begin{proof}[Proof of Proposition~\ref{propsTsx}] Let  $A:=\set{u_1,u_2,\dots,u_6}$. Figure~\ref{figequ6112},  according to Convention~\ref{conventionsgRsktWph}, indicates that we consider the following equivalences of $A$
\begin{equation}
\begin{aligned}
\alpha&=\kequ{u_4,u_5,u_6}, &\gamma&=\kequ{u_1,u_2,u_4}+\equ{u_3}{u_5}\cr
 \epsilon&=\kequ{u_1,u_2,u_3},&  \delta&=\kequ{u_1,u_3,u_6}+\equ{u_2}{u_5},
\end{aligned}
\end{equation}
and let $\beta:=\alpha+\epsilon$. Since $\alpha<\beta$, the set $X:=\set{\alpha,\beta,\gamma,\delta}$ is of order type $1+1+2$.
Let $S$ be the sublattice generated by $X$; we are going to show that $X=\Equ A$. Observe that
\begin{align*}
&\equ{u_2}{u_1}=\beta\gamma\in S, && \equ{u_1}{u_3}=\beta\delta\in S,\cr
&\equ{u_5}{u_4}=\alpha(\gamma+\equ{u_1}{u_3})\in S, &&
\equ{u_6}{u_5}=\alpha(\delta+\equ{u_2}{u_1})\in S,
\cr
&\equ{u_4}{u_2}=\gamma(\delta + \equ{u_5}{u_4})\in S, \text{ and} && \equ{u_3}{u_6}=\delta(\gamma+\equ{u_6}{u_5})\in S.
\end{align*}
Hence, Lemma~\ref{lemmaHamilt} applies to the circle $\tuple{u_1,u_3,u_6,u_5,u_4,u_2}$, and we obtain that all atoms of $\Part A$ are in $S$. But $\Part A$ is an \emph{atomistic lattice}, that is, each of its elements is the join of some atoms; this 
 completes  the proof of Proposition~\ref{propsTsx}.\end{proof}

\begin{proof}[Proof of Proposition~\ref{propfVjh}]
Unfortunately, we have no elegant proof. However, we have  computer programs\footnote{{These programs are available from the authors' websites}} that list all four-element generating sets of $\Part 5$; there are exactly 5305 such sets.  And we have another program that checks if these 5305 sets are antichains. The application of this program completes the proof.
\end{proof}

\section{A lower bound on the number of four-element generating sets}\label{sectlowbnd}
First of all, we extend Convention~\ref{conventionsgRsktWph} by the following one.

\begin{convention}\label{conventionnwSszGshmmM} In our figures what follow, every  horizontal  straight edge is $\alpha$-colored, even if  this is not always indicated. The straight edges of slope $-1$, that is the southeast-northwest edges, are $\beta$-colored while the straight edges with slope $1$, that is the southwest-northeast edges, are $\gamma$-colored. Finally, the \emph{solid} curved edges are $\delta$-colored.  (We should disregard the \emph{dashed} curved edges unless otherwise stated.)
Figure~\ref{figconv} helps to keep this convention in mind. 
Note that later in the paper, $\alpha$, \dots, $\delta$ will often be ``decorated'' by superscripts like $\vph$ but these decorations will not appear in our figures. For example,
Figure~\ref{figvrzZ} visualizes $\stzgb$ with $\beta$-colored edges rather than $\stzgb$-colored ones. 
\end{convention}


\begin{figure}[ht]
\centerline
{\includegraphics[scale=0.9]{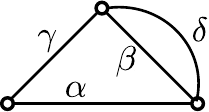}}
\caption{Notation for $\alpha$, $\beta$, $\gamma$ and $\delta$}
\label{figconv}
\end{figure}

Also, let us agree upon that
\begin{equation}\left.
\begin{aligned}
\sum_{\text{for all meaningful x}}\equ{u_x}{v_x} &\text{ will be denoted by}\cr
&\faequ{u_x}{v_x}\text{ or }\faequ{u_y}{v_y};
\end{aligned}\,\right\}
\end{equation}
that is, each of $x$ and $y$ in subscript or superscript position will mean that a join is formed for all meaningful values of these subscripts or superscripts. If only a part of the meaningful subscripts or superscripts are needed in a join, then the following notational convention will be in effect:
\begin{equation}
\xequ{u^{(i)}}{v^{(i)}}{i\in I}\quad\text{ stands for }\quad
\sum_{i\in I}\equ{u^{(i)}}{v^{(i)}}.
\end{equation} 
For an \emph{odd} positive integer $m$, to be referred to as \emph{Z-length}, define 
\begin{equation}
k=\zwid m:=\frac{m+3}2,\,\,\text{ to be referred to as \emph{Z-width}.}
\label{eqtxtHskzftmmCs}
\end{equation}
Equivalently, we can assume that the Z-width $k\geq 2$ is given; then 
\begin{equation}
m=\zlen k:=2k-3.
\label{eqtxtzhBmhQxlTgtNkf}
\end{equation}
The letter $Z$ in ``Z-length'' and ``Z-width'' has been chosen to remind us that an important particular case of what we are going to define is due to Z\'adori~\cite{zadori}.
Let us emphasize that $m$ (possibly with a subscript or superscript) in this paper is always an \emph{odd} positive integer.
Given $m$ and $k=\zwid m$ as above, a pair $\pair s t$ of nonnegative integers is called a \emph{necktie $($with respect to Z-length $m$ \textup{or} with respect to Z-width $k$)}  
if either $\pair s t=\pair 1 1$, or $0\leq s<t\leq k-1=\zwid m-1$. The pair $\pair 1 1 $ is called the \emph{trivial necktie} while we speak of a \emph{nontrivial necktie} if $s<t$. 
By a \emph{pin vector} (for a configuration to be defined soon) of Z-length $m$ we mean a vector $\vec x=\tuple{x_1,\dots,x_m}$ of $m$ bits.  That is, each $x_i$ is 0 or 1. In Section~\ref{sectdirprod}, it will be made clear why ``pin'' occurs in our terminology. Pin vectors are the same as bit vectors but pin vectors are used for a specific  purpose.  
The number of bits, $m$, is also called the \emph{dimension} of $\vec x$. 
By an \emph{identification quadruple $\vph$} or, in short, an \emph{id-quadruple} we mean a quadruple 
\begin{equation}\left.
\parbox{9.7cm}{ $\vph=\tuple{m,s,t,\vec z}$ 
such that $\vec z$ is an $m$-dimensional bit vector and either $\pair s t=\pair 1 1$, or $s$ and $t$ are integers satisfying  $0\leq s<t\leq \zwid m-1$. We will often denote
$m$, $s$, $t$, and $\vec z$ by $m_\vph$, $s_\vph$, $t_\vph$, and $\vec z_\vph$, respectively. Also, we are going to use the notation   $ k_\vph:=\zwid m$ and 
$n_\vph:=
\begin{cases}
m+4,&\text{if }\pair s t=\pair 1 1;\cr
m+5,&\text{otherwise}.
\end{cases}
$
}\,\,\,\right\}
\label{eqpbxHnstrNKhmrrDm}
\end{equation}

\begin{figure}[ht]
\centerline
{\includegraphics[scale=0.9]{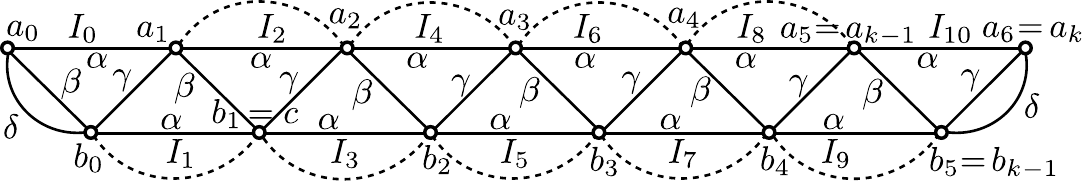}}
\caption{$\Zppfi{\tuple{9, 1,1,\vec 0}}$, that is, Z\'adori's thirteen element configuration (disregard the dashed arcs)}
\label{figzbrdd}
\end{figure}

The pair $\pair s t=\pair{s_\vph}{t_\vph}$ is the necktie of $\vph$. 
Note that $\vph$ is determined by the triple $\tuple{s_\vph,t_\vph,\vec z_\vph}$. However, we use quadruples rather than triples because, say, $\vph:=\tuple{3,1,2,\vec 0}$ is \emph{not} determined by  $\tuple{1,2,\vec 0}$ without specifying the dimension of the zero vector in it. If we need another id-quadruple, we usually denote it by
\begin{equation}\left.
\parbox{8.2cm}{ $\pvph=\tuple{m',s',t',\pvec z}=\tuple{m_\pvph,s_\pvph,t_\pvph,\pvec z}$, and we also use the notation $k'=k_\pvph=\zwid{m'}$ and $n'=n_\pvph$.
}\,\,\,\right\}
\label{eqpbxnhMrtnkVnsRlTn}
\end{equation}
Also, if $s_\vph\neq t_\vph$, then we define the 
\begin{equation}\left.
\parbox{8.9cm}{\emph{necktie-free version} $\mvph:=\tuple{m_\vph,1,1,\vec z_\vph}$ of $\vph$, and we use the notation $\mvph=\tuple{m_\mvph,s_\mvph,t_\mvph,\vec z_\mvph}=\tuple{m_\vph,1,1,\vec z_\vph}$,  $k_\mvph=\zwid{m_\mvph}$, and $n^-=n_\mvph=n_\vph-1$.
}\,\,\,\right\}
\label{eqpbxsznxsgbTrW}
\end{equation}
In other words, we obtain $\mvph$ from $\vph$ by changing the nontrivial necktie of $\vph$ to the trivial one.
Associated with an id-quadruple $\vph$ from \eqref{eqpbxHnstrNKhmrrDm}, we define an $n_\vph$-element system $\Zfi$ as follows; this system will be called a \emph{configuration} or, if $\vph$ needs to be specified, a \emph{$\vph$-configuration}. Letting $k:=k_\vph$, the  underlying set of  $\Zfi$ is  also denoted by $\Zfi$ and it consists of
\begin{equation}\left.
\begin{aligned}
&\kern 1.6cm \stza_0,\,\,\stza_1,\,\dots,\,\stza_k, \stzb_0,\,\stzb_1,\,\dots,
\stzb_{k-1},\,\stzc\cr
&\parbox{9.0cm}{such that the elements $\stza_0,\dots, 
,\stza_k, \stzb_0,\dots,\stzb_{k-1}$ are pairwise distinct,   $\stzc$ is distinct from the previous elements if $s_\vph<t_\vph$, and  $\stzc=\stzb_1$ if $\pair {s_\phi}{t_\phi}=\pair 1 1$; so $|\Zfi|=n_\phi$.}
\end{aligned}\,\,\right\}
\label{eqpbxbnLltkJlJlPl}
\end{equation}
Without their superscripts, these elements are visualized in Figures~\ref{figzbrdd} and \ref{figZgRjmCbsWs} for $\vph=\tuple{9,1,1,\vec 0}$, Figure~\ref{figzbrkkpe} for $\vph=\tuple{3,0,1,\vec 0}$, and Figure~\ref{figvrzZ} for $\vph=\tuple{9,1,4,\vec 0}$; disregard the dashed arcs (and arrows) in Figures~\ref{figzbrdd}, \ref{figzbrkkpe},  \ref{figvrzZ} and the dotted and dashed ovals in Figure~\ref{figZgRjmCbsWs}.
(The dashed arcs will be needed later.) The elements of $\Zfi$ are always placed in the plane in the way shown by the above-mentioned figures; this will be necessary to maintain Convention~\ref{conventionnwSszGshmmM}. 
 Note that regardless the value of the necktie $\pair s t:=\pair{s_\vph}{t_\vph}$,\, $\stzc$ is the intersection point of the line through $\stza_s$ and $\stzb_s$ with the line through $\stza_{t+1}$ and $\stzb_t$.
Let $m:=m_\vph$ and $k:=k_\vph$. The edges $\pair{\stza_0}{\stza_1}$, $\pair{\stzb_0}{\stzb_1}$, $\pair{\stza_1}{\stza_2}$, $\pair{\stzb_1}{\stzb_2}$, $\pair{\stza_2}{\stza_3}$, $\pair{\stzb_2}{\stzb_3}$, \dots , 
  $\pair{\stzb_{k-2}}{\stzb_{k-1}}$,  $\pair{\stza_{k-1}}{\stza_{k}}$, listed in zigzags,  
\begin{equation}
\text{will be denoted by $\stzI_0$, $\stzI_1$, $\stzI_2$, $\stzI_3$, $\stzI_4$, $\stzI_5$, \dots $\stzI_{m}$, $\stzI_{m+1}$,}
\label{eqtxtMhRskSlHgW} 
\end{equation}
respectively; see Figures~\ref{figzbrdd}--\ref{figzbrkkpe}, where the superscripts $\vph$ are omitted. The least equivalence collapsing the endpoints of $I_j$ will be denoted by $\kequ{I_j}$, that is, $\kequ{I_0}=\equ{\stza_0}{\stza_1}$,  $\kequ{I_1}=\equ{\stzb_0}{\stzb_1}$, etc. In order to turn $\Zfi$ into a structure, we define the following equivalences; $z_i$ will denote the $i$-th bit of $\vec z_\vph$:
\allowdisplaybreaks{
\begin{align}
&\begin{aligned}
  \stzga:=\kequ{&\stza_0,\stza_1,\dots \stza_k} +\kequ{\stzb_ 0,\stzb_1,\dots \stzb_{k-1}}\cr
 = \faequ{&\stza_x}{\stza_{x+1}}+\faequ{\stzb_y}{\stzb_{y+1}},
\end{aligned}
\cr
&\begin{aligned}
\stzgb:={}&\equ{\stzb_s}{\stzc} + \faequ{\stza_x}{\stzb_x}\cr
={}&\equ{\stzb_s}{\stzc} + \xequ{\stza_i}{\stzb_i}{0\leq i\leq k-1},
\end{aligned}
\cr
&\begin{aligned}
\stzgc:={}&\equ{\stzb_t}{\stzc} +  \faequ{\stza_{x+1}}{\stzb_x}\cr 
= {}&\equ{\stzb_t}{\stzc} +  \xequ{\stza_{i+1}}{\stzb_i}{0\leq i\leq k-1},
\end{aligned}
\cr
&\stzgd:=\equ{\stza_0}{\stzb_0}+\equ{\stza_k}{\stzb_{k-1}}
+\sum_{ {\scriptstyle i=1}  \atop {\scriptstyle z_i=1} }^{m} \kequ{I_i}
,\cr
&\text{and we let }\Zfi:=\tuple{\Zfi;\stzga,\stzgb,\stzgc,\stzgd}.
\label{eqZgRnGvSwsdr}
\end{align}
}%
So,  the $\vph$-configuration is defined by \eqref{eqZgRnGvSwsdr}. Let us emphasize that even if $I_0$ and $I_{m+1}$ are defined, they do not occur in the definition of $\stzgd$.  
According to Convention~\ref{conventionnwSszGshmmM}, 
the equivalences $\stzga$,  $\stzgb$, and $\stzgc$ are
visualized by Figures~\ref{figzbrdd}, \ref{figZgRjmCbsWs}, \ref{figzbrkkpe}, and \ref{figvrzZ}. For $\vec z=\vec 0$, $\stzgd$ is visualized only if we omit all the dashed arcs. However, if $\vec z\neq\vec 0$, then only some of the dashed arcs should be omitted and the rest should be changed to solid arcs. 
The particular case $\Zppfi{\tuple{m,1,1,\vec 0}}$ is the so-called \emph{Z\'adori Configuration} of \emph{odd size} $n_{\tuple{m,1,1,\vec 0}}=|\Zppfi{\tuple{m,1,1,\vec 0}}|=m+4$; this configuration was introduced (with different notation) in Z\'adori~\cite{zadori}. For this particular case, the following lemma is due to  Z\'adori~\cite{zadori}. For the even case, that is, for the case of nontrivial neckties, our $\Zfi$ is different from Z\'adori's one even when $\vec z_\vph=\vec 0$.

\begin{figure}[ht]
\centerline
{\includegraphics[scale=0.9]{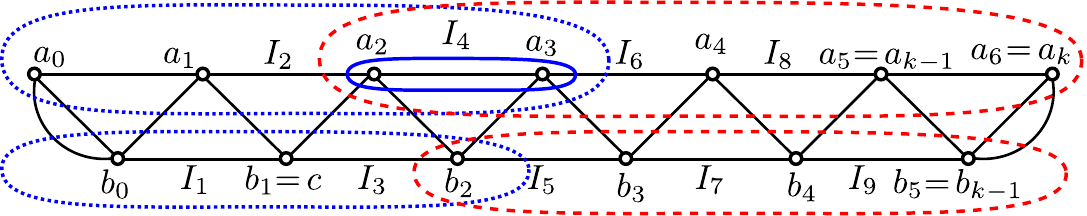}}
\caption{For $\vph=\tuple{9,1,1,\vec 0}$, the lower and upper bounds of $\stzf_4(\stzbmu)$ are illustrated by the solid oval and the dotted oval, respectively, while  $\rftrm 6(\stzbmu)$ is given by the dashed oval}
\label{figZgRjmCbsWs}
\end{figure}

\begin{lemma}\label{lemmaZoddxzG}
For every id-quadruple $\vph$, the equivalence lattice $\Equ{\Zfi}$ is generated by $\set{\stzga,\stzgb,\stzgc,\stzgd}$. Consequently, $\Part n$ is four-generated for all natural numbers $n\geq 5$.
\end{lemma}

\begin{proof}[Proof of Lemma~\ref{lemmaZoddxzG}]
With the quadruple $\obmu:=\tuple{\oal,\obe,\oga,\ode}$ of variables, we define the following quaternary terms by induction; even and odd subscripts will be denoted by  $2\mid i$ and $2\notdiv i$, respectively.
\begin{align}
\stzogb(\obmu)&:=\obe(\oal+\ode),\qquad  \stzogc(\obmu):=\oga(\oal+\ode)
\label{eqmCsltcGrna}\\
\stzf_0(\obmu)&:=\oal(\stzogb(\obmu)\cdot\ode+\stzogc), \quad\text{and for }i\leq m_\vph,
\label{eqZgbVrSdGnNzka}\\
\stzf_{i+1}(\obmu)&:=
 \begin{cases}
  \bigl(\stzf_i(\obmu)+\stzogb(\obmu)\bigr)\oal, &\text{if $2\mid i$, $i<m_\vph$, and $z_{i+1}=0$;}\cr
  \bigl(\stzf_i(\obmu)+\stzogb(\obmu)\bigr)\oal\ode, &\text{if $2\mid i$, $i<m_\vph$, and $z_{i+1}=1$;}\cr  
  \bigl(\stzf_i(\obmu)+\stzogc(\obmu)\bigr)\oal, &\text{if $2\notdiv i$, and $i=m_\vph$ or $z_{i+1}=0$;}\cr
  \bigl(\stzf_i(\obmu)+\stzogc(\obmu)\bigr)\oal\ode, &\text{if $2\notdiv i$, $i<m_\vph$,  and $z_{i+1}=1$.}
 \end{cases}\label{eqZgbVrSdGnNzkb}
\end{align}
Since $m_\vph$ is always odd, the inequality $i<m_\vph$ in the first two lines of \eqref{eqZgbVrSdGnNzkb} is only a redundant information with the purpose of increasing the readability of the proof. 
The condition in third line of  \eqref{eqZgbVrSdGnNzkb} abbreviates that $i$ is odd and, in addition, 
either $i=m_\vph$ (and then $z_{i+1}$ is undefined), or
$i<m_\vph$ and $z_{i+1}=0$.
Note our convention: \emph{overlined} greek letters stand for variables or terms or a tuple of variables; they are meaningful in all lattices not just in equivalence lattices.  
Note also that $\stzogb(\obmu)$ rather than $\obe$  in \eqref{eqZgbVrSdGnNzka} 
is redundant but this is on purpose, which will be clear after that \eqref{eqpbxKjtmsVstpFta}--\eqref{eqpbxKjtmsVstpFtc} are used. Lots of similar redundancy can be detected in the rest of this section. Furthermore, some inequalities and equalities stated  in this section will only be used in the next section; they are all obtained by trivial inductions.

We also define the following terms.
\begin{align}
\rftrm 0 (\obmu)&:= \oal\bigl(\stzogc(\obmu)\ode+  \stzogb(\obmu)\bigr)
\label{eqszKnjbRva}\\
\rftrm {i+1}(\obmu)&:=
\begin{cases}
  \bigl(\rftrm i(\obmu)+\stzogc(\obmu)\bigr)\oal, &\text{if $2\mid i$ ;}\cr
\bigl(\rftrm i(\obmu)+\stzogb(\obmu)\bigr)\oal, &\text{if $2\notdiv i$.}
\end{cases}
\label{eqszKnjbRvb}
\end{align}
Note that while 
\begin{equation}
\text{$\stzf_i(\obmu)$ is defined only for $0\leq i\leq m_\vph+1$,}
\label{eqtxthnGtvGcslfrlTdLmb}
\end{equation}
the terms $\rftrm i(\obmu)$ are defined for all $i\in\Nnul$. 
For convenience, we let
\begin{equation}
\stzbmu:=\tuple{\stzga,\,\,\stzgb,\,\,\stzgc,\,\,\stzgc  };
\label{eqMlKnySrDGttSz}
\end{equation}
see \eqref{eqZgRnGvSwsdr}. The restriction of an equivalence of $\Zfi$ to a subset $B\subseteq \Zfi$ will be denoted by $\restrict\epsilon B$; that is, $\restrict\epsilon B:=\set{\pair u v\in \epsilon: u\in B \text{ and }v\in B}$. After observing that
\begin{align}
\stzogb(\stzbmu)&=\restrict{\stzgb}{\set{\stza_0,\stza_1,\dots, \stza_{k_\vph}, \stzb_0,\stzb_1,\dots, \stzb_{k_\vph-1}}}\quad \text{ and}
\label{eqmgszRtmTBvsGa}\\
\stzogc(\stzbmu)&=\restrict{\stzgc}{\set{\stza_0,\stza_1,\dots, \stza_{k_\vph}, \stzb_0,\stzb_1,\dots, \stzb_{k_\vph-1}}}\,\,,
\label{eqmgszRtmTBvsGb}
\end{align}
a straightforward induction shows that for any   id-quadruples $\vph$ and $\pvph$, see \eqref{eqpbxHnstrNKhmrrDm} and  \eqref{eqpbxnhMrtnkVnsRlTn}, 
\begin{align}
&\stpzf_i(\stzbmu)\leq\kequ{\stzI_j:0\leq j\leq \min(i,m_\vph+1)}\,\, \text{ for }\,\, i\leq m_\pvph+1;
\label{eqBGkJggLjSjgygLsMz}
\\
&\text{in particular, }\stzf_i(\stzbmu)\leq
\kequ{\stzI_j:0\leq j\leq i}\,\, \text{ for }\,\, 0\leq i\leq  m_\vph+1.
\label{eqmCskszvhTScHrS}
\end{align}
Since we work with restrictions according to \eqref{eqBGkJggLjSjgygLsMz} and \eqref{eqmCskszvhTScHrS} and since we do not use any ``outer element'',  let us point out now and let us observe in what follows that 
\begin{equation}\left.
\parbox{9cm}{%
if the $(2k_\vph +1)$-element set
$\{\stza_0$, \dots, $\stza_{k_\vph}$, $\stzb_0$, \dots, $\stzb_{k_\vph-1}\}$ is a subset of a set $A$ and the \emph{restrictions} of $\stzga$, \dots, $\stzgd$ are the equivalences described in \eqref{eqZgRnGvSwsdr}, then
\eqref{eqBGkJggLjSjgygLsMz} and \eqref{eqmCskszvhTScHrS}  are  valid and the forthcoming equalities and inequalities up to   \eqref{eqzsbTdpsVFgDskwm} will be valid even in $\Equ A$.}
\,\,\,\right\}
\label{eqpbxbVmbHlmWzKsT}
\end{equation}
Similarly to \eqref{eqmCskszvhTScHrS}, we obtain by an easy induction that 
\begin{equation}
\kequ{\stzI_i} \leq  \stzf_i(\stzbmu)\,\, \text{ for }\,\, i\in\set{0,1,\dots, m_\vph+1}.
\label{eqmmMcnVgSkrgBrsbkMS}
\end{equation}
For a bit vector $\vec x=\tuple{x_1,x_2,\dots, x_{\dim(\vec x)}}$ and an integer $i\in\set{1,\dots, \dim(\vec x)}$, we define the 
$i$-dimensional \emph{initial segment} of  $\vec x$ as follows: 
\begin{equation}
\inseg{\vec x}i=\tuple{x_1,x_2,\dots, x_i}.
\label{eqZhghnTkzDjk}
\end{equation}
For $m$-dimensional bit vectors $\pvec x$ and $\ppvec x$, the inequality $\pvec x\leq \ppvec x$ is understood as $(\forall i\leq m)(x'_i\leq x''_i)$. For later reference, note  that if $m_\vph=\dim(\vec z_\vph)\leq \dim(\vec z_\pvph)=m_\pvph$, then the same induction as the one used for \eqref{eqmmMcnVgSkrgBrsbkMS} also shows the following.  
\begin{equation}\left.
\parbox{8.2cm}{Assume that $i\leq  m_\vph+1$ and  $i\leq  m_\pvph+1$. Assume also  that if $i=m_\vph+1$, then either $i=m_\pvph+1$, or   $i<m_\pvph+1$ and the $i$-th bit of $\vec z_\pvph$ is 0. Then 
$\inseg{\vec z_\pvph}{m_\vph}\leq \vec z_\vph$ implies that $
\kequ{\stzI_{i}} \leq  \stpzf_i(\stzbmu)$.}
\,\,\,\right\}
\label{eqpbxZhgGtrcshwnhSVkpL}
\end{equation}
For convenience at later references, we combine \eqref{eqmCskszvhTScHrS} and \eqref{eqmmMcnVgSkrgBrsbkMS} into
\begin{equation}
\kequ{\stzI_i} \leq \stzf_i(\stzbmu) \leq
\kequ{\stzI_j:0\leq j\leq i}\,\, \text{ for }\,\, i\in\set{0,1,\dots, m_\vph+1}.
\label{eqZrrtgrGth}
\end{equation}
An induction similar to what was needed for \eqref{eqmCskszvhTScHrS}  yields that 
\begin{equation}
\rftrm i(\stzbmu)=\kequ{\stzI_j: m_\vph+1-i\leq j\leq  m_\vph+1}\, \text{ for }\, i\in\set{0,1,\dots, m_\vph+1}.
\label{eqinTrwkhFt} 
\end{equation}
According to \eqref{eqZrrtgrGth} and \eqref{eqinTrwkhFt}, we think of the terms $\stzf_i$ and the terms $\rftrm i$ as the \emph{terms going to the right} and the \emph{terms going to the left}, respectively. This visual idea, which goes back to Z\'adori~\cite{zadori}, is important both in this section and in the next one. Equations \eqref{eqZrrtgrGth} and \eqref{eqinTrwkhFt} are illustrated in Figure~\ref{figZgRjmCbsWs}. 
Motivated by this figure, we define the following terms.
\begin{equation}
\stzg_{j}(\obmu):= \stzf_j(\obmu)\cdot {}\rftrm{ m_\vph+1-j}(\obmu),\quad\text{ for }\,\, j\in\set{0,1,\dots,m_\vph+1}.
\label{eqzhGhzkRptrtjRk}
\end{equation}
Combining \eqref{eqBGkJggLjSjgygLsMz} and \eqref{eqinTrwkhFt}, it follows that for any id-quadruples $\vph$ and $\pvph$ such that $m_\pvph=m_\vph$,
\begin{equation}
\stpzg_j(\stzbmu)\leq \kequ{\stzI_j}\quad\text{ for }\,\, j\in\set{0,1,\dots, m_\vph+1=m_\pvph+1}.
\label{eqskzTskzblGvfrPhkgKzWs}
\end{equation}
We can state even more for $\pvph=\vph$ since then  \eqref{eqmmMcnVgSkrgBrsbkMS} is also at our disposal to obtain the converse inequality. That is, 
it is clear by \eqref{eqZrrtgrGth} and \eqref{eqinTrwkhFt}, or by  Figure~\ref{figZgRjmCbsWs}, that
\begin{equation}
\stzg_j(\stzbmu)=\kequ{\stzI_j}\quad\text{ for }\,\, j\in\set{0,1,\dots, m_\vph+1}.
\label{eqdmRtnNdgBlW}
\end{equation}

Next, we define two additional quaternary terms, the \emph{side terms}
\begin{equation}
\sltrm{a_0,b_0}(\obmu):= \stzogb(\obmu) \cdot\ode\quad\text{and}\quad 
\srtrm{a_k,b_{k-1}}(\obmu):= \stzogc(\obmu) \cdot\ode.
\label{eqsDtrmSdf}
\end{equation}
Trivially,
\begin{equation}
\sltrm{\stza_0,\stzb_0}(\stzbmu):=\equ{\stza_0}{\stzb_0}\quad\text{and}\quad 
\srtrm{\stza_k,\stzb_{k-1}}(\stzbmu):=\equ{\stza_k}{\stzb_{k-1}}. 
\label{eqsDtrmSfwwjkQk}
\end{equation}
Now, depending on the necktie, the proof splits into two cases.

First, we assume that the necktie of $\Zfi$ is trivial, that is, $\pair{s_\vph}{t_\vph}=\pair 11$. With $k:=k_\vph$, let
\begin{equation}
\tuple{d_0,d_1,\dots,d_{t-1}}:=\tuple{\stza_0,\stza_1,\dots, \stza_{k}, \stzb_{k-1}, \stzb_{k-2},\dots, \stzb_0}.
\label{eqsdScprTrPnlKlKpcCtrB}
\end{equation}
Based on \eqref{eqdmRtnNdgBlW}, \eqref{eqsDtrmSfwwjkQk}, and \eqref{eqsdScprTrPnlKlKpcCtrB},  equality  \eqref{eqGbVrTslcdNssm} together with \eqref{eqtxthszKthTdtrDpcRvszQ} define a quaternary term $\keterm uv(\obmu)$ for any two distinct $u,v\in \Zfi$. This allows us to define a term $\stzeuv u v (\obmu)$ for any $u,v\in\Zfi$ as follows.
\begin{equation}
\stzeuv u v (\obmu):=
\begin{cases}
\text{the above-mentioned }\keterm u v(\obmu),&\text{if }u\neq v;\cr
\oal\cdot  \stzogb(\obmu)  \cdot  \stzogc(\obmu) \cdot\ode,&\text{if }u=v.
\end{cases}
\label{eqtxtszmflHzGnmHs}
\end{equation}
It follows from  \eqref{eqdmRtnNdgBlW}, \eqref{eqsDtrmSfwwjkQk}, and Lemma~\ref{lemmaHamilt} that
\begin{equation}
\stzeuv u v(\stzbmu)=\equ u v\quad\text{ for any }u,v\in\Zfi
\text{, provided }\vph\text{ is necktie-free}.
\label{eqzGrhksVgDbgrGlDm}
\end{equation}

Now let $\vph$ and $\pvph$ be arbitrary id-quadruples such that $m_\vph=m_\pvph$.  (In this paragraph, we do not have to assume that their neckties are trivial.)
Using the terms \eqref{eqsDtrmSdf}, which do not depend on $\vph$, and the terms occurring in \eqref{eqskzTskzblGvfrPhkgKzWs} rather than in \eqref{eqdmRtnNdgBlW}, and still based on the circle \eqref{eqsdScprTrPnlKlKpcCtrB} so that its superscripts are changed to $\pvph$, \eqref{eqtxtszmflHzGnmHs} defines a quaternary term $\stpzeuv u v(\obmu)$. Similarly to the previous paragraph, we obtain from   \eqref{eqskzTskzblGvfrPhkgKzWs}, \eqref{eqsDtrmSfwwjkQk}, and Lemma~\ref{lemmaHamilt} that
\begin{equation}\left.
\parbox{6.8cm}{
if $m_\vph=m_\pvph$, then 
$\,\stpzeuv u v(\stzbmu)\leq\equ u v\,$  for any $u,v$ in 
$\set{\stza_{0},\dots,\stza_{k_\vph},\stzb_{0},\dots,\stzb_{k_\vph-1} }$.}\,\,\,\right\}
\label{eqzsbTdpsVFgDskwm}
\end{equation}

 \begin{figure}[ht]
\centerline
{\includegraphics[scale=0.9]{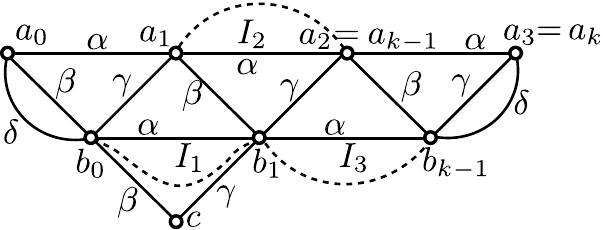}}
\caption{$\Zppfi{\tuple{3,0,1,\vec 0}}$}
\label{figzbrkkpe}
\end{figure}

Second, assume that the necktie $\pair s t:=\pair{s_\vph}{t_\vph}$ is nontrivial. Comment \eqref{eqpbxbVmbHlmWzKsT} allows us to use the terms associated with the necktie-free version $\mvph$. Let $k:=k_\vph=k_\mvph$. We can identify $\stmza_0$, \dots, $\stmza_k$, $\stmzb_0$, \dots, $\stmzb_{k-1}$ with 
$\stza_0$, \dots, $\stza_k$, $\stzb_0$, \dots, $\stzb_{k-1}$, respectively. That is, we consider $\Zmfi=\set{\stza_0, \dots, \stza_k, \stzb_0, \dots, \stzb_{k-1}}$ a subset of $\Zfi$. 
This allows us to let
\begin{equation}
\stzeuv {u}{v}(\obmu):= \stmzeuv u v (\obmu), \text{ for } u,v\in\Zmfi;\text{ see \eqref{eqtxtszmflHzGnmHs}.} 
\label{eqhmsHnZtzSpLkRznmTa}
\end{equation}
It follows form  \eqref{eqpbxbVmbHlmWzKsT} and \eqref{eqzGrhksVgDbgrGlDm} that 
\begin{equation}
 \stzeuv {u}{v}(\stzbmu)=\equ u v, \text{ for } u,v\in\Zmfi.
\label{eqhmsHnZtzSpLkRznmTb}
\end{equation}
Define
\begin{equation}\left. 
\begin{aligned}
\stzeuv {a_s}{c}(\obmu)&:= \obe\cdot \bigl( \oga+ \stzeuv  {a_s}{a_{t+1}}(\obmu) \bigr)=:\stzeuv c{a_s}(\obmu)\,\,\text{ and }\cr
\stzeuv {a_{t+1}}{c}(\obmu)&:= \oga\cdot \bigl( \obe+ \stzeuv {a_s}{a_{t+1}}(\obmu) \bigr) =:\stzeuv c{a_{t+1}}(\obmu).
\end{aligned}\,\,\right\}
\label{eqzghbjRwJF}
\end{equation}
It follows easily from \eqref{eqhmsHnZtzSpLkRznmTb} and the construction of $\Zfi$ that 
\begin{equation} 
\begin{aligned}
\stzeuv {a_s}{c}(\stzbmu)&= \stzeuv c{a_s}(\stzbmu)= \equ{\stza_s}{\stzc}\,\,\text{ and }\cr
\stzeuv {a_{t+1}}{c}(\obmu))&= \stzeuv c{a_{t+1}}(\obmu))= \equ{\stza_{t+1}}{\stzc}.
\label{eqmRknFltshJkD}
\end{aligned}
\end{equation}
Next, we define  a ``circle'' $\tuple{d_0,d_1,\dots,d_{t-1}}$ of \emph{all elements} of $\Zfi$ as follows, see the dashed arrows in Figure~\ref{figvrzZ}:
\begin{equation}\left.
\begin{aligned}
\tuple{d_0,&\dots,d_{t-1}}:=
\langle \stza_0,\stza_1,\dots, \stza_s, \stzc, \stza_{t+1}, \stza_{t+2},\dots,\cr
&\stza_k, \stzb_{k-1}, \stzb_{k-2},\dots,  
 \stzb_{t},\stza_t,\stzb_{t-1},\stza_{t-1},\stzb_{t-2},\cr
&\stza_{t-2},\dots,\stza_{s+1},\stzb_s,\stzb_{s-1},\dots, \stzb_0\rangle.
\end{aligned}\,\,\,\right\}
\label{eqsdzRpsndbblFnhThmThPdB}
\end{equation}
As \eqref{eqhmsHnZtzSpLkRznmTb} and \eqref{eqmRknFltshJkD} show, every two neighbouring elements of this circle have already been ``taken care of''. Hence, based on \eqref{eqhmsHnZtzSpLkRznmTb} and \eqref{eqmRknFltshJkD}, 
 equality \eqref{eqGbVrTslcdNssm},
 \eqref{eqtxthszKthTdtrDpcRvszQ}, and \eqref{eqsdzRpsndbblFnhThmThPdB} define a quaternary term 
\begin{equation}\keterm u v(\obmu)
\label{eqshzKnsvSkpJlRn}
\end{equation}  
for any two distinct $u$ and $v$ such that $\stzc\in\set{u,v}\subseteq \Zfi$. After letting $s=s_\vph$ and  $t:=t_\vph$, we let, for any $u,v\in\Zfi$,
\begin{equation}
\stzeuv u v (\obmu):= 
\begin{cases}
\stmzeuv u v (\obmu),&\text{if }u,v\in\Zmfi\text{ as in \eqref{eqhmsHnZtzSpLkRznmTa}};
\cr
\stzeuv u v (\obmu)\text{ from \eqref{eqzghbjRwJF}},&\text{if }\set{u,v}\in\set{\set{\stza_s,\stzc},\set{\stza_{t+1},\stzc}};\cr
\oal\cdot  \stzogb(\obmu)  \cdot  \stzogc(\obmu) \cdot\ode,&\text{if }u=v=c;
\cr
\keterm u v(\obmu)\text{ from \eqref{eqshzKnsvSkpJlRn}},&\text{otherwise.}
\end{cases}
\label{eqMssnzLlBkbsZwhgmLrRF}
\end{equation}
Note for later reference  that, due to the first line of \eqref{eqMssnzLlBkbsZwhgmLrRF},
\begin{equation}
\parbox{8.6cm}{for any $u,v\in\Zmfi$, the term $\stzeuv u v (\obmu)$ defined in \eqref{eqhmsHnZtzSpLkRznmTa}  and $\stzeuv u v (\obmu)$ defined in \eqref{eqMssnzLlBkbsZwhgmLrRF} are the same terms.}
\label{eqtxtsTsznflhLcSkh}
\end{equation}
We conclude from 
\eqref{eqhmsHnZtzSpLkRznmTb} and \eqref{eqtxtsTsznflhLcSkh}
(for $u,v\in\Zmfi$),  \eqref{eqmRknFltshJkD}, and Lemma~\ref{lemmaHamilt} that
\begin{equation}
\stzeuv u v(\stzbmu)={}\equ u v\quad \text{ for any }u,v\in\Zfi.
\label{eqzkfgHmrGlDmr}
\end{equation}
Due to \eqref{eqpbxbVmbHlmWzKsT} and  \eqref{eqtxtsTsznflhLcSkh}, \eqref{eqzGrhksVgDbgrGlDm} becomes a particular case of  \eqref{eqzkfgHmrGlDmr}. Therefore, 
\begin{equation}
\stzeuv u v(\stzbmu)={}\equ u v\quad \text{ for any id-quadruple $\vph$ and any }u,v\in\Zfi.
\label{eqzkfnbbBNhJdMr}
\end{equation}

\begin{figure}[ht]
\centerline
{\includegraphics[scale=0.9]{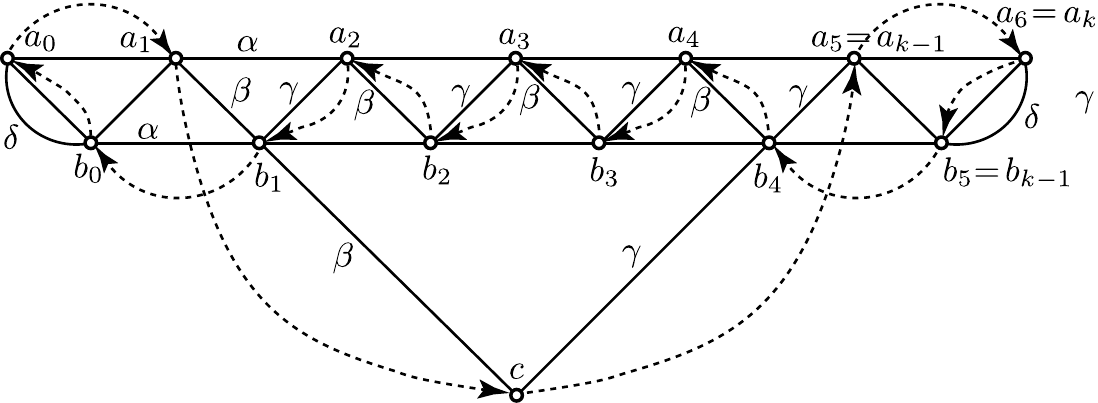}}
\caption{$\Zppfi{\tuple{9,1,4,\vec 0}}$; the dashed arrows indicate the ``circle'' \eqref{eqsdzRpsndbblFnhThmThPdB}}
\label{figvrzZ}
\end{figure}

Next, we interrupt the argument for the sake of some comments, which will be important only after the proof. 
First, observe that the terms $\stzeuv u v (\obmu)$ defined in \eqref{eqMssnzLlBkbsZwhgmLrRF}, see also \eqref{eqhmsHnZtzSpLkRznmTa} and  \eqref{eqtxtsTsznflhLcSkh}, depend on those defined in \eqref{eqtxtszmflHzGnmHs}. If we replace the terms in  \eqref{eqtxtszmflHzGnmHs} by those occurring in \eqref{eqzsbTdpsVFgDskwm}, which are terms yielding smaller values by the comparison of \eqref{eqzGrhksVgDbgrGlDm} and \eqref{eqzsbTdpsVFgDskwm}, then we clearly obtain terms  $\stpzeuv u v (\obmu)$ with smaller values. 
This proves the following counterpart of \eqref{eqzkfnbbBNhJdMr}:
\begin{equation}\left.
\parbox{9.5cm}{Let $\vph$ and $\pvph$ be id-quadruples such that $m_\vph=m_\pvph$, and let us identify $\stza_0,\dots, \stza_{k_\vph},\stzb_0,\dots,\stzb_{k_\vph-1}$ with 
 $\stpza_0,\dots, \stpza_{k_\pvph},\stpzb_0,\dots,\stpzb_{k_\pvph-1}$, respectively. We identify $\stzc$ and $\stpzc$ if and only if $\pair{s_\vph}{t_\vph}=\pair{s_\pvph}{t_\pvph}$
Then $\stpzeuv u v(\stzbmu)\leq\equ u v$ holds in $\Equ\Zfi$  for any $u,v\in\Zfi\cap\Zmfi$.
}\,\,\,\right\}
\label{eqszGnslmsLmjTzZbh}
\end{equation} 
Note that  $\set{\stza_0,\dots, \stza_{k_\vph},\stzb_0,\dots,\stzb_{k_\vph-1}}\subseteq \Zfi\cap\Zpfi$ holds above,  due to the identification of elements; of course, $k_\vph=k_\pvph$ also holds.

In order to define \emph{covering terms}, assume  till the end of the present proof that
no parentheses are omitted from  lattice terms just because the binary lattice operations are associative. That is, we cannot omit the parentheses from, say $(x_1x_2)x_3$. Let $\tau_1$ and $\tau_2$ be subterms of a lattice term $\tau$.  If  $\tau_1$ is a proper subterm of  $\tau_2$ but $\tau_1$ is not a proper subterm of any proper subterm of $\tau_2$, then we say that $\tau_2$ is the {covering term} of $\tau_1$ in $\tau$. Any proper subterm of $\tau$ has a unique covering term (in $\tau$). The term $\tau$ itself has no covering term but this case will not occur below. 
For later reference, let us observe that
the terms occurring in this proof can be parenthesized
in a straightforward way so that in these terms,
\begin{align}&
{\left.\parbox{7cm}{the only covering terms of  $\ode$ are the following four: $\oal+\ode$, $\oal\ode$, 
$\stzogb(\obmu)\ode$,
and  $\stzogc(\obmu)\ode$;}\,\,\,\right\}}
\label{eqpbxKjtmsVstpFta}
\\
&
{\left.\parbox{7cm}{apart from the trivial terms $\keterm u u$ and $\stzeuv u u$,   $\stzogb(\obmu) \ode$ is a subterm only in \eqref{eqZgbVrSdGnNzka} and \eqref{eqsDtrmSdf} while $\stzogc(\obmu) \ode$  only in \eqref{eqszKnjbRva} and \eqref{eqsDtrmSdf}; and}\,\,\,\right\}}
\label{eqpbxKjtmsVstpFtb}
\\
&\text{if $\vec z_\vph=\vec 0$, then $\oal\ode$ does not occur as a subterm.}
\label{eqpbxKjtmsVstpFtc}
\end{align}
For example, we parenthesize the term occurring in the second line of \eqref{eqZgbVrSdGnNzkb} as
$\bigl(\stzf_i(\obmu)+\stzogb(\obmu)\bigr)\cdot(\oal\ode)$.


Finally, resuming the proof,    \eqref{eqzkfnbbBNhJdMr} 
 implies that $\sublat{\stzga,\stzgb,\stzgc,\stzgd}$ contains all atoms of $\Equ\Zfi$. Thus, we conclude  the statement of Lemma~\ref{lemmaZoddxzG}.
\end{proof}

For $n\in\NN$, the number of partitions of the set $\set{1,2,\dots,n}$, that is $|\Equ n|$, is the so-called $n$-th \emph{Bell number}; see, for example,  Rennie and Dobson~\cite{renniedobson}. The $n$-th Bell number will be denoted by $\Bell n$.
The \emph{upper} and \emph{lower integer parts} of a real number will denoted by $\upint x$ and $\loint x$, respectively; for example, $\upint{\sqrt 2\,}=2$ and $\loint{\sqrt 2}=1$.
We define the \emph{width} $\wid n$ and the \emph{length}  $\len n$ of an integer $n\geq 5$ as follows:
\begin{equation}
\wid n:=\loint{(n-1)/2}\qquad \text{ and } \qquad\len n:=\begin{cases}
n-4,&\text{if }2\notdiv n,\cr
n-5,&\text{if }2\mid n;
\end{cases}
\label{eqkrSnmkGtmKlPtrZg}
\end{equation}
they should not be confused with the Z-width and Z-length defined in \eqref{eqtxtHskzftmmCs} and  \eqref{eqtxtzhBmhQxlTgtNkf}. Note that $\len n= \zlen{\wid n}$ and $\wid n= \zwid{\len n}$.

\begin{proposition}\label{proplwStmtsH}
For $n\geq 7$, $\Part n$ has at least
\begin{equation}
\frac 12\cdot n!\cdot 
{{\wid n} \choose 2}^{n-4-\len n}  \cdot \Bell {\wid n-1}  \cdot \Bell {\wid n}
\label{eqTrzRsslbRsnkbnTh}
\end{equation}
four-element generating sets; see  \eqref{eqkrSnmkGtmKlPtrZg} for the notation. 
\end{proposition}

Note that, for $n\in\set{4,5,6}$, the number $\gnu n$ of four-element generating sets of $\Part n$ will exactly be given in Section~\ref{sectstatcomp}. The case $n\in\set{1,2,3}$ is trivial and will not be considered.
Note also the following.

\begin{remark} In the proof of Proposition~\ref{proplwStmtsH}, we will explicitly construct \eqref{eqTrzRsslbRsnkbnTh} many four-element generating sets. Combining the methods of the present paper with any of the earlier papers dealing with four-element generating sets of $\Equ n$ or, equivalently, $\Part n$, one can construct even more. 
For example, one can ``resect'' finite parts of the infinite constructions used in Cz\'edli~\cite{czedli4gen,czsmall,czedli112gen}, and one can even modify them in various ways. But all these possible improvements in our horizon would need \emph{much} more work than  \eqref{eqTrzRsslbRsnkbnTh} without yielding a ``nice'' explicit lower estimate and without getting close to the truth suspected by Section~\ref{sectstatcomp}. Hence, here we only prove as much as \emph{conveniently} possible with the toolkit developed for the sake of our main target, Theorem~\ref{thmmain}. Proving an estimate better than \eqref{eqTrzRsslbRsnkbnTh} is postponed to a forthcoming paper. 
\end{remark}

\begin{proof}[Proof of Proposition~\ref{proplwStmtsH}] 
Let $k=\wid n$ and $m:=\len n$. Pick an id-quadruple $\vph=\tuple{m,s,t,\vec 0}$ such that  $\Zfi$ from \eqref{eqZgRnGvSwsdr} consists of $n$ elements. Note that if $n$ is odd, then $\pair s t=\pair 1 1$. To ease the notation in this proof, we drop the superscript $\vph$ of the elements of $\Zfi$; for example, $a_1$ and $b_0$ stand for $\stza_1$ and $\stzb_0$. (In this way, Figure~\ref{figzbrdd} is in full harmony with the present proof.) Clearly, 
\begin{equation}\left.
\parbox{9cm}{there are $\Bell {k-1}  \cdot \Bell {k}$ ways to select a pair $\tuple{\mu_1,\mu_2}$ from $\Equ{\set{a_1,a_2,\dots,a_{k-1}}} \times \Equ{\set{b_0,b_1,\dots,b_{k-1}}}$.
}\,\,\right\}
\label{eqtxtZhgShsVNrSczhhs}
\end{equation}
For each pair $\pair{\mu_1}{\mu_2}$ from \eqref{eqtxtZhgShsVNrSczhhs}, let $\kequ{\mu_1}$ denote the equivalence of $\Zfi$ generated by $\mu_1$; this makes sense since $\mu_1$ is a subset of $\Zfi\times \Zfi$. The meaning of  $\kequ{\mu_2}$ is analogous. We let 
\begin{equation}
\delta'= \delta'(\mu_1,\mu_2):=\kequ{\mu_1}+ \equ{a_0}{b_0}+ \kequ{\mu_2}+\equ{a_{k}}{b_{k-1}},
\label{eqhSmkfNftcntNmrsnKspMk}
\end{equation}
and we take $\stzga$,  $\stzgb$, and $\stzgc$ from \eqref{eqZgRnGvSwsdr}.
Observe also that
\eqref{eqZrrtgrGth} for $i=0$ and 
 \eqref{eqsDtrmSfwwjkQk} remain valid if we replace $\stzgd$ in $\stzbmu$ by $\delta'$. Therefore, in virtue  of \eqref{eqpbxKjtmsVstpFta}--\eqref{eqpbxKjtmsVstpFtc}, we conclude that 
\begin{equation}
G(\tuple{\mu_1,\mu_2}:=\set{\stzga,\stzgb,\stzgc,\delta'=\delta'(\mu_1,\mu_2)}
\label{eqZhGrBWgpkwvlrsG}
\end{equation}
is a four-element generating set. 
We claim that 
\begin{equation}
\text{ $G(\tuple{\mu_1,\mu_2})$ determines the pair $\tuple{\mu_1,\mu_2}$}.
\label{eqtxtnbSdjTmbwsT}
\end{equation}
In order to see this, note that $G(\tuple{\mu_1,\mu_2})$ is only a set and it might be unclear at first sight which of its four members is, say, $\stzga$. However, $\stzga$ can be recognized by the property that it is the only member of $G(\tuple{\mu_1,\mu_2})$ that

\begin{equation}\left.
\parbox{8.4cm}{has exactly two nonsingleton blocks and one of these blocks is $(k+1)$-element while the other is  $k$-element; furthermore, it has at most one singleton block}
\,\,\,\right\}
\label{eqtxtHztBmszRGfKr}
\end{equation}
Indeed, property \eqref{eqtxtHztBmszRGfKr} fails for  $\stzgb$ and $\stzgc$ since each of them has a two-element block. (Note that for $2\mid n$, we have just  used the assumption  $n\geq 7$ and so $k>2$.) 
Since at least one of the $\delta'$-blocks is a subset of $\set{a_1,\dots,a_{k-1}}$, 
\eqref{eqtxtHztBmszRGfKr} does not hold for $\delta'$; either because $\delta'$ has $k-1$ singleton blocks, or because $\delta'$ has an at most $(k-1)$-element nonsingleton block.
(We have used  $n\geq 7$ again, which implies that $k-1>1$.)
Thus, $G(\tuple{\mu_1,\mu_2})$ determines $\stzga$. Next, consider the following property of an equivalence $\epsilon\in \Equ\Zfi$:
\begin{equation}\left.
\parbox{6.5cm}{$\epsilon\neq\stzga$ and, in addition, either $\epsilon\stzga\neq \enul$,  or 
$\epsilon\stzga=\enul$ and 
$\epsilon$ has a singleton block.}
\,\,\,\right\}
\label{eqskzhGrnPBstTsln}
\end{equation}
Since only $\delta'$ out of $\set{\stzga,\stzgb,\stzgc,\delta'}$ has this property,  $\delta'$ is also recognized.
We obtain $\mu_2$ as the restriction of $\delta'$ to the smaller nonsingleton $\stzga$-block. Hence, $\mu_2$ is determined by $G(\tuple{\mu_1,\mu_2})$.
From the largest $\stzga$-block, delete 
\begin{equation}\left.
\parbox{9cm}{the unique element $x$ such that there exists an element $y$ in the smaller nonsingleton $\stzga$-block such that $\pair x y\in\stzgb \delta'$;}
\,\,\right\}
\label{eqpbxZghCskjhNQvCR}
\end{equation}
this element is $a_0$. (This identification of $a_0$ will be used later.) 
Replacing $\stzgb$ in \eqref{eqpbxZghCskjhNQvCR} with $\stzgc$, we can describe $a_k$; delete it too. There remains $\set{a_1,\dots, a_{k-1}}$. Although $\stzgb$ and $\stzgc$  are not determined by $G(\tuple{\mu_1,\mu_2})$ separately, the set $\set{\stzgb,\stzgc}$ is determined since so are $\stzga$ and $\delta'$. 
Hence,  $\set{a_1,\dots, a_{k-1}}$ is also determined, and so is $\mu_1$, which is the restriction of $\delta'$ to this set. 
Therefore, \eqref{eqtxtnbSdjTmbwsT} holds.

Next, we assume that $n$ is odd until the opposite is explicitly said. 
We obtain from 
\eqref{eqtxtZhgShsVNrSczhhs} and \eqref{eqtxtnbSdjTmbwsT} that for $\Zfi$, that is, for a given labelling of the elements $\set{1,2,\dots,n}$ by 
$\Zfi$ and defining $\stzga$, $\stzgb$, $\stzgc$, and $\delta'$ accordingly, there are $\Bell k  \cdot \Bell {k-1}$  many generating sets of the form \eqref{eqZhGrBWgpkwvlrsG}. 
The permutations of $\Zfi$ give new generating sets.
Two distinct permutations $\pi_1$ and $\pi_2$ of $\Zfi$ can give the same generating set. But this happens if and only if $\pi_1\pi_2^{-1}$ belongs to the ``stabilizer''
\[\begin{aligned}
\Stb{\Zfi}:= \bigl\{\pi\in \Sym{\Zfi}: \set{\pi(\stzga),\pi(\stzgb),\pi(\stzgc),\pi(\delta')} \cr=  \set{\stzga,\stzgb,\stzgc,\delta'}\bigr\}.
\end{aligned}
\]
Let us fix the following terminology:  for a subset $X$ and a relation $\rho$ of $\Zfi$, a permutation $\pi\in\Sym\Zfi$ is said to \emph{preserve} $X$ and to \emph{preserve} $\rho$ if $\set{\pi(u):u\in X}=X$ and  $\set{\pair{\pi(u)}{\pi(v)}: \pair u v\in\rho } =\rho$, respectively.
Assume that $\pi\in\Stb \Zfi$. Since $\stzga$ is the only member of $\set{\stzga,\stzgb,\stzgc,\delta'}$ having property \eqref{eqtxtHztBmszRGfKr}, $\pi$ preserves $\stzga$. Since the $\stzga$-blocks are of different sizes, each of them is preserved by $\pi$. Using property \eqref{eqskzhGrnPBstTsln}, we obtain that $\pi$ preserves $\delta'$. Hence, $\pi$ also preserves $\set{\stzgb,\stzgc}$.  So
\begin{equation}\left.
\parbox{6.1cm}{$\pi$ preserves each of $\set{a_0, a_1,\dots,a_k}$, $\set{b_0, b_1,\dots,b_{k-1}}$, $\stzga$, $\delta'$, and $\set{\stzgb,\stzgc}$.}\,\,\right\}
\label{eqhzGnRmnRmg}
\end{equation}
Combining \eqref{eqpbxZghCskjhNQvCR}, the property obtained from \eqref{eqpbxZghCskjhNQvCR} by replacing $\stzgb$ by $\stzgc$, and \eqref{eqhzGnRmnRmg}, it follows that $\set{\set{a_0,b_0},\set{a_k,b_{k-1}}}$ is preserved and, additionally,  
\begin{align}
\text{either }&\pi(a_0)=a_k,\,\, \pi(a_k)=a_0,\,\,\pi(b_0)=b_{k-1}, \,\,\pi(b_{k-1})=b_0,
\label{eqbhzTmmRngjsKsLta}
\\
\text{or }&\pi(a_0)=a_0,\,\,  \pi(a_k)=a_k,\,\, \pi(b_0)=b_{0},\,\, \pi(b_{k-1})=b_{k-1}.
\label{eqbhzTmmRngjsKsLtb}
\end{align}
For $j\in\NN$, let 
\[\chi_j:=(\underbrace{\stzgb\circ\stzgc\circ\stzgb\circ\dots}_{j\text{ factors}})\cup (\underbrace{\stzgc\circ\stzgb\circ\stzgc\circ\dots}_{j\text{ factors}}),
\]
where $\circ$ denotes relational product; for example, 
$\stzgb\circ\stzgc=\set{\pair xy: (\exists z)\bigl(\pair xz\in\stzgb\text{ and }\pair zy\in\stzgc \bigr)}$. 
Since $\set{\stzgb,\stzgc}$ is preserved,  $\pi\in\Stb \Zfi$ preserves $\chi_j$ for all $j\in\NN$. 
For $j=1,2,\dots$, the element $a_j$ is characterized by the property that $\pair{a_0}{a_j}$ belongs to $(\stzga\cap \chi_{2j})\setminus\chi_{2j-1}$.  Observe that $\pi$ preserves this property since it preserves $\stzga$,  $\chi_{2j}$ and $\chi_{2j-1}$. In particular, for $j\in\set{2,3,\dots, k-1}$,
\begin{equation}
\pair{\pi(a_0)}{\pi(a_j)}\in (\stzga\cap \chi_{2j})\setminus\chi_{2j-1}.
\label{eqZghsmThClFgrchmmsC}
\end{equation}
Now if \eqref{eqbhzTmmRngjsKsLta}, then
\eqref{eqZghsmThClFgrchmmsC} yields that $\pi(a_j)=a_{k-j}$
for all meaningful $j$. It follows similarly that $\pi(b_j)=b_{k-1-j}$. Hence, if \eqref{eqbhzTmmRngjsKsLta}, then $\pi$ is the reflection across the vertical symmetry axis of $\Zfi$, see Figure~\ref{figzbrdd}. It follows similarly that \eqref{eqbhzTmmRngjsKsLtb} implies that $\pi$ is the identity permutation. We have obtained that at most two permutations belong to $\Stb{\Zfi}$: the  reflection mentioned above and the identity permutation. 
So $|\Stb{\Zfi}|\leq 2$.  Hence, Lagrange's Theorem gives that  $\Stb{\Zfi}$ has at least $\Zfi!/2=n!/2$ cosets. Permutations from different cosets send $\set{\stzga,\stzgb,\stzgc,\delta'}$ to different generating sets; this explains the $n!/2$ in \eqref{eqTrzRsslbRsnkbnTh}. Since $n$ has been assumed to be odd,  $n-4-\len n=0$ by \eqref{eqkrSnmkGtmKlPtrZg} and the  binomial coefficient in \eqref{eqTrzRsslbRsnkbnTh} disappears. Therefore, the argument above proves the statement for $n$ odd.

From now on till the end of the proof, we assume that $n$ is even and we let $n':=n-1$. 
There are $n$ ways to pick an element $c=\stzc$, which will be the unique singleton $\stza$-block.  It is clear by the construction that if $\set{\stzga,\stzgb,\stzgc,\delta'}$ is a generating set
of $\Equ\Zfi$, then the restrictions of these four equivalences generate the lattice of equivalences of $\Zfi\setminus\set{\stzc}$; compare Figures~\ref{figzbrkkpe} and \ref{figvrzZ} with 
Figure~\ref{figzbrdd}. Since the odd case has already been settled, there are  
\begin{equation}
\frac 1 2\cdot n'!\cdot \Bell{\wid {n'}-1}\cdot \Bell{\wid {n'}}
\label{eqdzhBldnnbRnjlkRk}
\end{equation}
ways to pick the four-element set of these restricted equivalences.
To each such set, there are 
\begin{equation}
{\wid{n'}}\choose 2
\label{eqZrcmrKzgnmRZrD}  
\end{equation}
ways to add a necktie. If two four-element sets of the above-mentioned restricted equivalences are distinct ``without neckties", then they will remain distinct after ``putting on'' neckties. 
So the number of four-element generating sets we have shown up is the product of $n$,  \eqref{eqdzhBldnnbRnjlkRk}, and \eqref{eqZrcmrKzgnmRZrD}. 
But $\wid n=\wid{n'}$, $n\cdot n'!= n!$, and $n-4-\len n=1$, so this product is \eqref{eqTrzRsslbRsnkbnTh}. This completes the proof of  Proposition~\ref{proplwStmtsH}.
\end{proof}

For a few values of $n$, the number described by  \eqref{eqTrzRsslbRsnkbnTh}
is given in Table~\ref{tableYSkszjSnhmHL}. For comparison, note that, say, $16!$ is only $2.092\cdot 10^{13}$, up to rounding.
\isitneeded{
\begin{table}
\[
\vbox{\tabskip=0pt\offinterlineskip
\halign{\strut#&\vrule#\tabskip=2pt plus 2pt&
#\hfill& \vrule\vrule\vrule#&
\hfill#&\vrule#&
\hfill#&\vrule#&
\hfill#&\vrule#&
\hfill#&\vrule#&
\hfill#&\vrule\tabskip=0.1pt#&
#\hfill\vrule\vrule\cr
\vonal\vonal\vonal\vonal
&&\hfill$n$&&$7$&&$8$&&$9$&&$10$&&$11$&
\cr\vonal\vonal
&&\hfill\eqref{eqTrzRsslbRsnkbnTh}&&$25\,200$&&$604\,800$&&$13\,608\,000$&&$816\,480\,000$&&$15\,567\,552\,000$&
\cr\vonal
&&\hfill$100p_{\eqref{eqTrzRsslbRsnkbnTh}}$&&$1.03\cdot 10^{-4}$&&$4.95\cdot 10^{-6}$&&$1.63\cdot 10^{-7}$&&$1.08 \cdot 10^{-8}$&&$1.76 \cdot 10^{-10}$&
\cr\vonal
\vonal\vonal\vonal
}} 
\]
\caption{The estimate \eqref{eqTrzRsslbRsnkbnTh} for $n\in\set{7,\dots,10}$}\label{tablekszjSnhmHL}
\end{table}
}

\begin{table}\[
\vbox{\tabskip=0pt\offinterlineskip
\halign{\strut#&\vrule#\tabskip=2pt plus 2pt&
#\hfill& \vrule\vrule\vrule#&
\hfill#&\vrule#&
\hfill#&\vrule#&
\hfill#&\vrule#&
\hfill#&\vrule#&
\hfill#&\vrule\tabskip=0.1pt#&
#\hfill\vrule\vrule\cr
\vonal\vonal\vonal\vonal
&&\hfill$n$&&$7$&&$8$&&$9$&&$10$&&$11$&
\cr\vonal
&&\hfill\eqref{eqTrzRsslbRsnkbnTh}&&$25\,200$&&$604\,800$&&$13\,608\,000$&&$816\,480\,000$&&$15\,567\,552\,000$&
\cr\vonal\vonal\vonal\vonal
  \isitneeded{&&\hfill$100p_{\eqref{eqTrzRsslbRsnkbnTh}}$&&$1.03
  \cdot 10^{-4}$&&$4.95\cdot 10^{-6}$&&$1.63\cdot 
  10^{-7}$&&$1.08 \cdot 10^{-8}$&&$1.76 \cdot 10^{-10}$&
  \cr\vonal}
&&\hfill$n$&&$12$&&$13$&&$14$&&$15$&&$16$&
\cr\vonal
&&\hfill\eqref{eqTrzRsslbRsnkbnTh}&&$1.868\cdot 10^{12}$&&$3.287\cdot 10^{13}$&&$6.902\cdot 10^{15}$&&$1.164\cdot 10^{17}$&&$3.911\cdot 10^{19}$&
\cr\vonal\vonal
\vonal\vonal
}} 
\]
\caption{The estimate \eqref{eqTrzRsslbRsnkbnTh} for $n\in\set{7,8,\dots,16}$}\label{tableYSkszjSnhmHL}
\end{table} 

\isitneeded{
\begin{table}
\[
\vbox{\tabskip=0pt\offinterlineskip
\halign{\strut#&\vrule#\tabskip=2pt plus 2pt&
#\hfill& \vrule\vrule\vrule#&
\hfill#&\vrule#&
\hfill#&\vrule#&
\hfill#&\vrule#&
\hfill#&\vrule#&
\hfill#&\vrule\tabskip=0.1pt#&
#\hfill\vrule\vrule\cr
\vonal\vonal\vonal\vonal
&&\hfill$n$&&$12$&&$13$&&$14$&&$15$&&$16$&
\cr\vonal\vonal
&&\hfill\eqref{eqTrzRsslbRsnkbnTh}&&$1.868\cdot 10^{12}$&&$3.287\cdot 10^{13}$&&$6.902\cdot 10^{15}$&&$1.164\cdot 10^{17}$&&$3.911\cdot 10^{19}$&
\cr\vonal
&&\hfill$100p_{\eqref{eqTrzRsslbRsnkbnTh}}$&&$1.42\cdot 10^{-11}$&&$1.35\cdot 10^{-13}$&&$1.25\cdot 10^{-14}$&&$7.64\cdot 10^{-17}$&&$7.78\cdot 10^{-18}$&
\cr\vonal
\vonal\vonal\vonal
}} 
\]
\caption{The estimate \eqref{eqTrzRsslbRsnkbnTh} for $n\in\set{12,\dots,16}$}\label{tablekjfzHfVL}
\end{table}
}

\isitneeded{
\begin{table}
\[
\vbox{\tabskip=0pt\offinterlineskip
\halign{\strut#&\vrule#\tabskip=2pt plus 2pt&
#\hfill& \vrule\vrule\vrule#&
\hfill#&\vrule#&
\hfill#&\vrule#&
\hfill#&\vrule#&
\hfill#&\vrule\tabskip=0.1pt#&
#\hfill\vrule\vrule\cr
\vonal\vonal\vonal\vonal
&&\hfill$n$&&$98$&&$99$&&$100$&&$1000$&
\cr\vonal\vonal
&&\hfill\eqref{eqTrzRsslbRsnkbnTh}&&$1.252\cdot 10^{245}$&&$3.15\cdot 10^{246}$&&$3.70\cdot 10^{251}$&&$5.288\cdot 10^{4252}$&
\cr\vonal
&&\hfill$100p_{\eqref{eqTrzRsslbRsnkbnTh}}$&&$3.17\cdot 10^{-203}$&&$1.10\cdot 10^{-207}$&&$1.73\cdot 10^{-208}$&&$1.59\cdot 10^{-3454}$&
\cr\vonal
\vonal\vonal\vonal
}} 
\]
\caption{The estimate \eqref{eqTrzRsslbRsnkbnTh} for some large $n$}\label{tablehjfcRszd}
\end{table}
}

\section{Generating direct products}\label{sectdirprod}
The goal of this section is to prove  the only theorem of this paper. Before the reader gets frightened by the technicalities that are needed to formulate its part (B) and the complicated  nature of this part, we formulate two of its corollaries.

\begin{corollary}\label{coroltconsecutive} 
For every integer number $n\geq 9$, 
\begin{equation}
\Part n\times \Part{n+1}\times\dots\times \Part{3n-14}
\label{eqmhTnrSzmdgKszMzHb}
\end{equation}
is a four-generated lattice.
\end{corollary}

\begin{corollary}\label{corolmanyfactors} For every positive integer $u$, 
\begin{equation}
\Part{n}^u\times \Part{n+1}^u\times\dots \Part{n+u-1}^u
\label{eqhgsvMtmSkpbNhsrLVt}
\end{equation} 
is a four-generated lattice for all sufficiently large integers $n$.
\end{corollary}

As a further appetizer for our theorem, we note the following.

\begin{remark}\label{remarkcZgBnjT} Let $L_1$, \dots, $L_t$ be lattices. If the direct product $L_1\times\cdots\times L_t$ is four-generated, then for each  $i\in\set{1,\dots,t}$ and for every $J\subseteq \set{1,\dots, t}$, the lattice  $L_i$ and the direct product $\prod_{j\in J}L_j$ are at most four-generated.
Also (and in particular), if the direct product $L_1^{s_1}\times\cdots\times L_t^{s_t}$  of direct powers is four-generated, $J\subseteq \set{1,\dots,t}$, and $r_j\leq s_j$ for all $j\in J$, then  $\prod_{j\in J}L_j^{r_j}$ is  four-generated, too.
\end{remark}

We know from Z\'adori~\cite{zadori} that $\Part n$ is not 
three-generated for $n\geq 5$. Hence, 
Remark~\ref{remarkcZgBnjT} gives that  our theorem  implies the Strietz--Z\'adori result stating the $\gnu n$ from \eqref{pbxgnujl} is at least 1 provided $5\leq n\in\NN$. 
Since the natural projection homomorphisms $L_1\times\dots\times L_t\to L_i$ and  $L_1\times\dots\times L_t\to \prod_{j\in J}L_j$ send  generating sets to  generating sets, Remark~\ref{remarkcZgBnjT} is trivial and needs no  separate proof.
Yet we make another remark before introducing some notation that are needed to formulate our theorem. In Cz\'edli~\cite{czgparallel}, where only some direct powers of partition lattices were proved to be four-generated, each of the  configurations that gave rise to the direct factors were of the same width. Hence, we knew that when the values of the terms going to the right just abutted those of the terms going to the left in one of the direct factors, then the same happened in all direct factors. Now the difficulty is bigger since 
we have to defend our argument from the threat that the values mentioned above just abut in one direct factor but fully overlap in another factor.

The \emph{airiness} $\air {\vec x}$ of a bit vector $\vec x=\tuple{x_1,x_2,x_3,\dots, x_t}$ is the largest integer $u$ such that there are $u$ consecutive zeros among the $x_1,x_2,\dots, x_t$. That is, if there is a $j\in\set{1,\dots, t-u+1}$ such that $x_j=x_{j+1}=\dots =x_{j+u-1}=0$. If $\vdim{\vec x}$ is fixed, then $\vec 0:=\tuple{0,0,\dots,0}$ has the maximal arity, $\dim(\vec x)$, while $\air{\tuple{1,1,\dots,1}}=0$. 
Bit vectors of a fixed \emph{dimension}  form a boolean lattice with respect to the componentwise ordering.
For positive integers
$u$ and $t$, 
\begin{equation}\left.
\parbox{8.2cm}{let $\DBV u t$ be the set of all bit vectors of the form $\tuple{1,x_2,x_3,\dots, x_t}$ with airiness strictly less than $u$.}\,\,\right\}
\label{eqpbxZFUdmszmznflkKbL}
\end{equation}
The acronym comes from ``Dense Bit Vectors'' as we think of bit vectors with small airiness as dense ones. 
Since we are going to use large antichains of the poset $\DBV u t$, we define
\begin{equation}
\sba u t:=\max\set{ |X|: X\text{ is an antichain in }\DBV u t };
\label{eqsstsmsNklpdRrvsK}
\end{equation}
the acronym comes from ``Size of one of the Biggest Antichains''.
We also define
\begin{align}
\parbox{9cm}{
$\TRA u t:=\{\pair X Y: X$  and $Y$ are antichains in\\
\phantom{mmmm} $\DBV u t$  
and $(\forall\vec x\in X)\,(\forall \vec y\in Y)(\vec x\not\leq\vec y)\}$ and
}
\label{eqxzbtSWslSGqlRdkLLvN}
\\
\parbox{9cm}{
$\tra u t:=\{\pair {|X|} {|Y|}: \pair X Y \in \TRA u t\}$;
}\label{eqtrLhsjfzltgyl}
\end{align}
here the acronym comes from ``Two Related Antichains''.
A systematic study of $\sba u t$ and $\tra u t$ with involved tools of combinatorics is not targeted in the present paper; we  only give some easy facts that will be needed later. Namely, 
\begin{align}
&\sba u t   \geq {{t-1}\choose{u-1}}\,\;\text{ if }u-1\leq\upint{(t-1)/2},\label{eqcztsmsTrhltmRa}\\
&\sba u t  = {{t-1}\choose{\upint{(t-1)/2}}}\,\,\text{ if }u-1\geq \upint{(t-1)/2},
\label{eqSprNrskzPnnmDlrCjsX}\\
&\left\langle{{t-1}\choose{i-1}} , {{t-1}\choose{{i}}}\right\rangle \in  \tra u t\,\,\text{ for }i\in\set{1,\dots, \min(u,t)-1},
\label{eqcztsmsTrhltmRb}\\
&\parbox{7.5cm}{if $\pair {x}{y}\in\tra u t$, $0\leq x'\leq x$, and $0\leq y'\leq y$,\\ \phantom{mmmmmmmmmmmm}  then $\pair {x'}{y'}\in\tra u t$,}
\label{eqpbxmNkSgtSnkMhG}
\\
&\sba 1 t=1\text{ and }\set{\pair 01,\pair 10 }\subseteq \tra 1 t\text{ for }t\geq 1, 
\label{eqhmghZtkkrLnspSg}
\\
&\sba u 2=1\text{ and }\pair 1 1 \in  \tra u 2\text{ for }u\geq 2\text{, and} \label{eqcskknwgmLmnxFnT}\\
&\text{if $0<u_1<u_2$ and $0<t$, then $\tra {u_1} t \subseteq \tra  {u_2} t$.}
\label{eqnHspjzlkHMkLbR}
\end{align}
Indeed, \eqref{eqcztsmsTrhltmRa} follows from 
the fact that the bit vectors $\tuple{1,x_2,\dots, x_t}$ with exactly $u-1$ zero entries belong to $\DBV u t$ and  form an antichain. For $i\in\set{1,\dots, \min(u,t)-1}$, the bit vectors $\tuple{1,x_2,\dots, x_t}$ with exactly $i-1$ zero entries and those with exactly $i$ zero entries form an antichain $X_{i-1}$ and an antichain $X_i$, respectively, and 
 $\pair{|X_{i-1}|}{|X_{i}|} \in \tra u t$ implies  \eqref{eqcztsmsTrhltmRb}. 
Let $B_{t-1}$ be the boolean lattice of all $(t-1)$-dimensional bit vectors, and let $B'_{t-1}:=\set{\tuple{1,x_2,\dots,x_t}: \tuple{x_2,\dots,x_t}\in B_{t-1}}$. The map $B_{t-1}\to B'_{t-1}$, defined by 
$\tuple{x_2,\dots,x_t}\mapsto \tuple{1,x_2,\dots,x_t}$ is an order isomorphism. We know from Sperner's Theorem, proved by Sperner~\cite{sperner} and cited in Gr\"atzer~\cite[page 354]{gratzer}, that the $(t-1)$-dimensional bit vectors with exactly $\upint{(t-1)/2}$ zero entries form a maximum sized antichain $X$ in $B_{t-1}$. Let $X'$ be the image of $X$ with respect to the just-mentioned isomorphism. Clearly, $X'$ is a maximum-sized antichain in $B'_{t-1}$. Since 
$X'\subseteq \DBV u t\subseteq B'_{t-1}$ and $|X'|=|X|={{t-1}\choose{\upint{(t-1)/2}}}$, we conclude \eqref{eqSprNrskzPnnmDlrCjsX}. The validity of \eqref{eqpbxmNkSgtSnkMhG} is trivial. 
In order to see the validity of \eqref{eqhmghZtkkrLnspSg}, note that $\DBV 1 t=\set{\vec1}$ is a singleton and both $\DBV 1 t$ and $\emptyset$ are antichains. Since, for $u\geq 2$,  $\DBV u 2=\set{\pair 1 0,\pair 1 1}$ is the two-element chain, \eqref{eqcskknwgmLmnxFnT} is trivial. Finally,
\eqref{eqnHspjzlkHMkLbR} is trivial again since its premise yields that  $\DBV {u_1} t\subseteq \DBV {u_2} t$.

Using the notation given in \eqref{eqsstsmsNklpdRrvsK} and
\eqref{eqtrLhsjfzltgyl}, we are in the position to formulate our main theorem.

\begin{theorem}[Main Theorem]\label{thmmain}\ 

\textup{(A)} If $5\leq n<n'\in\NN$, then the direct product $\Part{n}\times \Part {n'}$ of the partition lattices $\Part {n}$ and $\Part{n'}$ is a four-generated lattice.
 
\textup{(B)} Let $d$ be an \emph{odd} positive integer, and let $m_1,m_2,m_3,\dots, m_{d+2}$  also be \emph{odd} positive integers such that $d\leq m_1<m_2<m_3<\dots <m_{d+2}$ and $m_1\geq 3$.
Also, for $i\in\set{1,2,\dots,d+2}$, let 
\begin{equation}
\wextra_i={{(m_i+3)/2}\choose{2}},
\label{eqlGnwsqjxTr}
\end{equation}
and  let $p_i$ and  $q_i$ be nonnegative integers such that $p_i+q_i>0$ and 
\begin{align}
&\left.
\begin{aligned}
  &p_1+q_1=\sba{d}{m_1-1}\text{, see \eqref{eqsstsmsNklpdRrvsK} and  \eqref{eqcztsmsTrhltmRa}--\eqref{eqcskknwgmLmnxFnT},  or }\cr
  &\pair{p_1}{q_1}\in\tra {d}{m_1-1},\text{ see  \eqref{eqtrLhsjfzltgyl} and  \eqref{eqcztsmsTrhltmRb}--\eqref{eqcskknwgmLmnxFnT},} 
\end{aligned} 
\,\,\right\}\,\,
\label{eqthmmaina}\\ 
&\left.
\begin{aligned}
  &p_2+q_2=\sba{m_2-d}{m_2-d-1}\text{,  or }\cr
  &\pair{p_2}{q_2}\in\tra {m_2-d}{m_2-d-1},
\end{aligned}\,\,\right\}
\label{eqthmmainb}
\end{align}
and for each $i\in \set{3,4,\dots,d+2}$, 
\begin{align}
&p_i+q_i=\sba {d+3-i} {m_i-d-1},\text{   or}
\label{eqthmmainc}\\
&\pair{p_i}{q_i}\in\tra {d+3-i}  {m_i-d-1}.
\label{eqthmmaind}\end{align}
Finally, for $i\in\set{1,\dots,d+2}$, let 
\begin{equation}
\text{$n_i':=m_i+4$ and $n_i'':=m_i+5$.}
\label{eqhCshbfhRvhswkm} 
\end{equation}
Then 
\begin{equation} 
\prod_{i=1}^{d+2} \Bigl( \Part{n_i'}^{p_i}\times \Part{n_i''}^{q_i\wextra_i} \Bigr),
\label{eqdzGrlsTmTnhzk}
\end{equation}
which is a direct product of direct powers of partition lattices, is a four-generated lattice. 
\end{theorem}

Note that \eqref{eqdzGrlsTmTnhzk} can also be written as 
$\prod_{i=1}^{d+2}  \Part{n_i'}^{p_i}\times \prod_{i=1}^{d+2} \Part{n_i''}^{q_i\wextra_i}$. Before proving our theorem, we illustrate its meaning by an example.

\begin{example}\label{example2020}
While part (B) of Theorem~\ref{thmmain} does not give much for small partition lattices, it gives  more for large ones if not too many of them are considered. For example, part (B) implies that 
\begin{equation}
\prod_{n=1011}^{2020}\Part n^{{10^{127}}}
\label{eqdpXtbSlmBwWpS}
\end{equation}
is a four-generated lattice.
To see this, let $d:=579$ and let $m_i:= 1005 +2i$ for $i\in\set{1,2,\dots, 505}$. We do not use the $m_i$ for $i=506,\dots, d+2$; in other words, for $i\in\set{506, \dots, d+2}$, we pick the pair $\pair {p_i}{q_i}=\pair 0 0$ from 
$\tra {d+3-i}  {m_i-d-1}$, see \eqref{eqpbxmNkSgtSnkMhG}--\eqref{eqhmghZtkkrLnspSg} and \eqref{eqthmmaind}. (See also Remark~\ref{remarkcZgBnjT} since the theorem itself does not permit that $\pair {p_i}{q_i}=\pair 0 0$.) Using the estimates \eqref{eqcztsmsTrhltmRa}--\eqref{eqSprNrskzPnnmDlrCjsX} and the equations in \eqref{eqthmmaina}--\eqref{eqthmmainc} and letting $p_i=q_i$ for $i\in\set{1,2,\dots, 505}$, we obtain by computer algebra (namely, by a Maple worksheet available from the first author's website)   that all these $p_i$ and $q_i$ are at least $10^{127}$. Note that the initial choice $d:=579$ was not a random one because other choices would have yielded smaller exponents.
By Remark~\ref{remarkcZgBnjT},  we can change $w_i$ from \eqref{eqlGnwsqjxTr} to 1. Therefore, it follows from 
Theorem~\ref{thmmain} that the lattice given in \eqref{eqdpXtbSlmBwWpS} is four-generated.
\end{example}

\begin{proof}[Proof of Theorem~\ref{thmmain}] In addition to earlier experience with generating sets of partitions lattices, see the References section, an idea is also borrowed from locksmithing. In a traditional pin tumbler cylinder lock,
there are \emph{pins} of different (to be more precise,  not necessarily equal) lengths, and a key can turn the plug if and only if the heights of its \emph{ridges} match these lengths.
The $\sum \set{\kequ {\stzI_i}: x_i=1,\,\,1\leq i \leq m}$ part in the definition of $\stzgd$, see \eqref{eqZgRnGvSwsdr},  will play a role similar to that of the pins in a cylinder lock. The is why we call $\vec z_\vph$ a pin vector. The meetands $\ode$ in \eqref{eqZgbVrSdGnNzkb} will play the role of the heights of the above-mentioned ridges. Unfortunately, the situation in the proof is not as pleasant as in locksmithing. Namely, while a single ridge of wrong size of a real key fully prevents the plug of a real lock from turning around, an extra meetand in a wrong place paralyzes the sequence of our terms only for a moment in general and further meetands
are necessary to prevent the ``progress'' to the right later.

Let us enlighten the situation from another aspect. When we use a real key with code vector $\vec x$ to open a real pin tumbler cylinder lock with pin  vector $\vec y$, then the key and the lock are usually assumed to have the same ``length'', that is, 
$\dim(\vec x)=\dim(\vec y)$. If not, then  $\inseg{\vec x}{\dim(\vec y)}$ has the main effect; see \eqref{eqZhghnTkzDjk} for this notation. For a real key, possibly longer than needed, the necessary and sufficient condition to open the lock is the \emph{equality} 
$\inseg{\vec x}{\dim(\vec y)}=\vec y$. In our ``\emph{virtual locksmithing}'', the corresponding  necessary and sufficient condition  is, very roughly saying, 
\begin{equation}
\text{the \emph{inequality} $\,\inseg{\vec x}{\dim(\vec y)}\leq \vec y$;}
\label{eqtxthmBhZrsSJmLs}
\end{equation}  
note that \eqref{eqpbxZhgGtrcshwnhSVkpL}, 
\eqref{eqtxtZghthnGwpndJCn},  \eqref{eqpbxBlRSblJsNg},
\eqref{eqbRmrmlnDrbRsrNfsT},
\eqref{eqpbxBsmsgzsghWq}, and  \eqref{eqpbxmCHglsbrgSk}
are or will be connected to  \eqref{eqtxthmBhZrsSJmLs}.
The connection of \eqref{eqthmmaina}--\eqref{eqthmmaind} in the theorem with antichains occurring in \eqref{eqsstsmsNklpdRrvsK}--\eqref{eqnHspjzlkHMkLbR} is explained by our intention to exclude the strict inequality from \eqref{eqtxthmBhZrsSJmLs}; this is why our construction is involved.  Below, in the first part of the proof, we deal with ``virtual locksmithing". Let $\vph$ be an id-quadruple; see \eqref{eqpbxHnstrNKhmrrDm}. 

For a $\vph$-configuration $\Zfi$, defined in \eqref{eqZgRnGvSwsdr}, and a term $\stpzf_i$ going to the right, defined in \eqref{eqZgbVrSdGnNzka}--\eqref{eqZgbVrSdGnNzkb} (with $\vph$ instead of $\pvph$), 
\begin{equation}
\text{we say that $\stpzf_i$\,\, \emph{gets through}\,\, $\Zfi\,$ if $\,\stpzf_i({\stzbmu})\geq \kequ{\stzI_{m_\vph+1}}$,}
\label{eqtxTGtsThrgH}
\end{equation}
where $\stzbmu$ and $\kequ{\stzI_{m_\vph+1}}$ have been defined in \eqref{eqMlKnySrDGttSz} and in (and after) \eqref{eqtxtMhRskSlHgW}, respectively. 
It follows immediately from \eqref{eqBGkJggLjSjgygLsMz} that 
\begin{equation}
\text{if $\stpzf_i$ gets through $\Zfi$, then $i\geq m_\vph+1$.}
\label{eqtxtZghthnGwpndJCn}
\end{equation}

Let us agree that when we deal with bit vectors, then $0^i$ and $1^i$ will denote the $i$-dimensional zero vector $\tuple{0,\dots,0}$ and unit vector $\tuple{1,\dots,1}$, respectively. Also,  for $\vec x=\tuple{x_1,\dots,x_i}$ and 
$\vec y=\tuple{y_1,\dots,y_j}$, the concatenation $\tuple{x_1,\dots,x_i,y_1,\dots,y_j}$ is denoted by 
$\vec x\lev \vec y$. With this convection, we can denote, for example, the eleven-dimensional vector $\tuple{1,0,0,0,1,1,0,1,1,1,1}$ by $1\lev 0^3\lev 1^2\lev 0\lev 1^4$ and an arbitrary vector beginning with two 0 components by $0^2\lev \vec z$.
Note that the concatenation with the ``empty'' vector $0^0=1^0$ makes no effect; for example, $1^0\lev \vec x=\vec x =\vec x\lev 0^0$.

For a while, let $\vph$ with all of its ingredients mentioned in \eqref{eqpbxHnstrNKhmrrDm} as well as $\pvph$ from \eqref{eqpbxnhMrtnkVnsRlTn} be fixed. In particular, $m=m_\vph$, $\vec z=\vec z_\vph$, and $\pvec z=\vec z_\pvph$.
As a counterpart of \eqref{eqtxtZghthnGwpndJCn}, we assert that 
for any nonnegative integers  $i\leq m_\pvph+1$ and  $p$  and any positive integer $q$, 
\begin{equation}\left.
\parbox{7.4cm}{if $\vec z$ is of the form $\vec z= 1^p\lev 0^q\lev \vec y$\, and $\air{\pvec z}<q$, then $\stpzf_i$ does not get through  $\Zfi$,}
\,\,\right\}
\label{eqpbxBlRSblJsNg}
\end{equation}
no matter how large $m'=m_\pvph$ and  $i\leq m'+1$ are. 
Note that the possibility $\dim(\vec y)=0$, that is $\dim(\vec z)=p+q$, is allowed. 
In order to  show the validity of  \eqref{eqpbxBlRSblJsNg}, assume that $0\leq j\leq i$. We are going to define  $\eff j$, the \emph{effectiveness} of a subscript $j$ with respect to $\stpzf_i$ and $\Zfi$
or, shortly,   with respect to \eqref{eqpbxBlRSblJsNg}, as follows. 
Let $\eff j$ be 
\begin{equation}
\text{the smallest $u$ such that }0\leq 
u\leq m+1 \text{ and } \stpzf_j(\stzbmu) \leq \sum_{\othv=0}^u \kequ{\stzI_{\othv}}.
\label{eqZhgmnHfLGz}
\end{equation}
By  \eqref{eqBGkJggLjSjgygLsMz}, $\eff j$ exists and 
\begin{equation}
\eff j\leq \min(j,m+1). 
\label{eqczgShRdkTsChszlGRk}
\end{equation}
Visually, we can think of $\eff j$ as the ``distance'' that measures ``how far the equivalence relation $\stpzf_j(\stzbmu)$ has gone'' from $\kequ{\stzI_0}$. 
Since $m=\dim(\vec z)=\dim(1^p\lev 0^q\lev \vec y)\geq p+q$, \eqref{eqpbxBlRSblJsNg} will clearly follow from the way we defined our concepts in \eqref{eqtxTGtsThrgH} and \eqref{eqZhgmnHfLGz} if we manage to prove that  
\begin{equation}
\text{if $\air{\pvec z}<q$ and $j\leq i$, then $\eff j<p+q$  with respect to \eqref{eqpbxBlRSblJsNg}.}
\label{eqpjbWhGhTjfhN}
\end{equation}
Indeed, \eqref{eqpjbWhGhTjfhN}, which has not been proved yet,  implies   \eqref{eqpbxBlRSblJsNg} by letting $j:=i$. 
Similarly to the straightforward proof of \eqref{eqBGkJggLjSjgygLsMz}, it follows easily that for $0\leq u\leq m+1$ and $j\leq m'$,
\begin{equation}
\stpzf_j(\stzbmu)   \leq \sum_{\othv=0}^u \kequ{\stzI_{\othv}}  \text{  implies that  } \stpzf_{j+1}(\stzbmu)  \leq  \sum_{\othv=0}^{\min(u+1,m+1)} \kequ{\stzI_{\othv}}. 
\label{eqrzhGhfkdtmszlDtm}
\end{equation}
This yields that, for all $j\in\NN$ such that $j < i\leq m'+1$,
\begin{equation}
\eff{j+1}\leq \eff j +1.
\label{eqhTknSnTnzKtsVknLflC}
\end{equation}

Now, for the sake of contradiction, suppose that  \eqref{eqpjbWhGhTjfhN} fails. Then we know that  $\air{\pvec z}<q$ but
\begin{equation}
\text{there is a $j$ such that  $j\leq i$ and $\eff j \geq p+q$.}
\label{eqFjtNkmnnRfmrGhnstvsG}
\end{equation}
We know from \eqref{eqhTknSnTnzKtsVknLflC} that  effectiveness increases ``slowly'' and so it ``cannot jump over numbers''. This reminds us a continuous $\mathbb R\to\mathbb R$ function; later, we will refer to this property of effectiveness as ``continuity''. 
Continuity and \eqref{eqFjtNkmnnRfmrGhnstvsG} yield a least subscript $u$ with $\eff u=p$. Similarly, there exists a 
least subscript $w$ with $\eff w=p+q$. As a consequence of \eqref{eqhTknSnTnzKtsVknLflC}, note that 
\begin{equation}
\text{for any $u'<u$ and $w'<w$, $\eff{u'}<p$ and $\eff{w'}<p+q$.}
\label{eqtxthmsrznPrsqWvkXccVrZ}
\end{equation}
In particular,   $u<w$. Note another consequence of \eqref{eqhTknSnTnzKtsVknLflC}: $w-u\geq q$, so  
\begin{equation}
w-q\geq u\geq 0\,\, \text{ and }\,\,   w-(q-1)\geq u+1\geq 1.
\label{eqsHcSnRjTkPsTmJfdtNgh}
\end{equation}
We claim that, for every $v < w$,
\begin{equation}
\text{if $\eff {v}<p+q$ and $z'_{v+1}=1$, then $\eff{v+1}\leq p$.}
\label{eqtxthfjbrkzlKrhGv}
\end{equation}
Since $\eff {v}\leq p+q-1$, we have that  $\stpzf_v(\stzbmu)   
\leq \sum\set{\kequ{\stzI_\othv}: \othv\leq p+q-1}$. Combining this inequality with \eqref{eqrzhGhfkdtmszlDtm}, we obtain that 
$\stpzf_{v+1}(\stzbmu)   
\leq \sum\set{\kequ{\stzI_\othv}: \othv\leq p+q}$.
From $z'_{v+1}=1$ and  \eqref{eqZgbVrSdGnNzkb}, we obtain that 
$\stpzf_{v+1}(\stzbmu)   \leq \stzgd$. The two inequalities we have just obtained yield that  
\begin{equation}
\stpzf_{v+1}(\stzbmu)   \leq \stzgd \cdot  \sum\set{\kequ{\stzI_\othv}: \othv\leq p+q}.
\label{eqlMsndtRrhDkp}
\end{equation}
The nonsingleton blocks of the second meetand on the right of \eqref{eqlMsndtRrhDkp} are  
\begin{align}
&\kequ{\stzI_0}+\kequ{\stzI_2}+\dots+\kequ{\stzI_{2\loint{(p+q)/2}}}=
\kequ{a_0,a_1,\dots,a_{1+\loint{(p+q)/2}}} \text{ and}
\label{eqZhgRhBmSrZszpdfa}
\\
&\kequ{\stzI_1}+\kequ{\stzI_3}+\dots+\kequ{\stzI_{2\loint{(p+q+1)/2}-1}}=
\kequ{b_0,b_1,\dots,b_{\loint{(p+q+1)/2}}}.
\label{eqZhgRhBmSrZszpdfb}
\end{align}
So these nonsingleton blocks are initial segments of the upper horizontal line and the lower horizontal line in Figure~\ref{figzbrdd}. Note that $\set{2\loint{(p+q)/2},\,2\loint{(p+q+1)/2}-1 } =\set{p+q-1, p+q}$.
By $z_{p+1}=z_{p+2}=\dots=z_{p+q}=0$ and the construction of  $\Zfi$, see  \eqref{eqZgRnGvSwsdr}, $\stzgd$ collapses none of $\stzI_{p+1}$, \dots, $\stzI_{p+q}$, which are on the right of the horizontal segments given in \eqref{eqZhgRhBmSrZszpdfa} and \eqref{eqZhgRhBmSrZszpdfb}.  This implies that the inequality $\othv\leq p+q$ in
\eqref{eqlMsndtRrhDkp} can be replaced by $\othv\leq p$. 
Hence, $\stpzf_{v+1}(\stzbmu)   \leq  \sum\set{\kequ{\stzI_\othv}: \othv\leq p}$. Therefore,  $\eff{v+1}\leq p$, proving \eqref{eqtxthfjbrkzlKrhGv}.

Next, iterating \eqref{eqhTknSnTnzKtsVknLflC}, we obtain the 
\begin{equation}
\eff{j+r}\leq \eff j +\othv\text{ whenever }0\leq j\leq j+\othv\leq i\leq m'+1.
\label{eqsrKkcsmTgnGjsRt}
\end{equation}
We claim that 
\begin{equation}
p<\eff{w-\othv}\,\,\text{ for all }\,\, \othv\in\set{0,1,\dots, q-1};
\label{eqbRmrmlnDrbRsrNfsT}
\end{equation}
note that $w-\othv>0$ by \eqref{eqsHcSnRjTkPsTmJfdtNgh}.
To verify \eqref{eqbRmrmlnDrbRsrNfsT}, suppose the contrary. Then there is an $\othv\in\set{0,1,\dots, q-1}$ such that 
\[
p+q=\eff w=\eff{(w-\othv)+r}\leqref{eqsrKkcsmTgnGjsRt} \eff{w-\othv}+\othv\leq
p+\othv<p+q,
\]
which is a contradiction. This shows the validity of \eqref{eqbRmrmlnDrbRsrNfsT}. 

Next, for $\othv\in\set{0,1,\dots, q-1}$, let $v_\othv:= w-\othv-1$; note that $v_\othv\geq 0$ by  \eqref{eqsHcSnRjTkPsTmJfdtNgh}.
Since $\eff{v_\othv+1}=\eff{w-\othv}>p$ by  \eqref{eqbRmrmlnDrbRsrNfsT}, $v_\othv$ fails to satisfy the conclusion of  \eqref{eqtxthfjbrkzlKrhGv}. Hence, $v_\othv$ fails to satisfy the premise of  \eqref{eqtxthfjbrkzlKrhGv} either. However, we know from \eqref{eqtxthmsrznPrsqWvkXccVrZ} and $v_\othv<w$ that  $\eff{v_\othv}<p+q$. 
Hence the only way that the premise of  \eqref{eqtxthfjbrkzlKrhGv} fails is that $z'_{v_\othv+1}=0$. This gives us $q$ consecutive zeros in $\pvec z$ since $\othv$ ranges in $\set{0,1,\dots, q-1}$. So $\air{\pvec z}\geq q$. This inequality contradicts the $\air{\vec z}<q$ part of \eqref{eqpjbWhGhTjfhN}, which we have assumed. Thus, we have proved  \eqref{eqpjbWhGhTjfhN}.  In virtue of the sentence following   \eqref{eqpjbWhGhTjfhN}, we have also proved \eqref{eqpbxBlRSblJsNg}.

For the particular case  $\pair pq=\pair 01$, a stronger variant of \eqref{eqpbxBlRSblJsNg} holds; namely, we claim that 
for any nonnegative integer $i$ and any positive integer $q$, 
\begin{equation}\left.
\parbox{9.7cm}{if $\vec z=\vec z_\vph$  and $\pvec z=\vec z_\pvph$ are of the forms
 $\vec z=  0\lev \vec y$\, and  $\pvec z=  1\lev \pvec y$,   respectively, then $\stpzf_i$ does not get through  $\Zfi$,}
\,\,\right\}
\label{eqpbxBsmsgzsghWq}
\end{equation}
provided $i\leq m_{\pvph}+1$ (since otherwise $\stpzf_i$ is undefined). In order to verify this, observe that $\stzgd$ does not collapse $\stzI_1$ since the first bit of $\vec z_\vph$ is 0, see \eqref{eqZgRnGvSwsdr}. Also, $\stzgd$ fails to collapse $\stzI_0$, regardless what $\vph$ is. Hence,
$(\kequ{\stzI_0}+\kequ{\stzI_1})\cdot \stzgd=\enul$, where $\enul$ (still) denotes the smallest element of $\Equ{\Zfi}$. 
This together with 
\eqref{eqBGkJggLjSjgygLsMz} (in which the role of $\vph$ and that of $\pvph$ are now interchanged) and \eqref{eqZgbVrSdGnNzkb} give that $\stpzf_1(\stzbmu)=\enul$. Since both $(\enul +\stzgb)\stzga$ and  $(\enul +\stzgc)\stzga$  (even without the meetand $\stzgd$) are equal to $\enul$,
a trivial induction based on  \eqref{eqZgbVrSdGnNzkb} yields that 
$\stpzf_i(\stzbmu)=\enul$ for all meaningful $i$. This implies  \eqref{eqpbxBsmsgzsghWq}.

Next, we claim that for arbitrary id-quadruples $\vph$ and $\pvph$, the following holds; the notation given in \eqref{eqpbxHnstrNKhmrrDm} and \eqref{eqpbxnhMrtnkVnsRlTn} will be in effect.
\begin{equation}\left.
\parbox{6.7cm}{Assume that $m:=m_\vph=m_\pvph$. Then 
$\stpzf_{m+1}$  gets through  $\Zfi$ if and only if  $\pvec z\leq \vec z$.}
\,\,\,\right\}
\label{eqpbxmCHglsbrgSk}
\end{equation}
In order to prove \eqref{eqpbxmCHglsbrgSk}, note that we already know from \eqref{eqpbxZhgGtrcshwnhSVkpL} that $\pvec z\leq \vec z$ implies that $\stpzf_m$  gets through  $\Zfi$.
Hence, it suffices to deal with the ``only if'' part of \eqref{eqpbxmCHglsbrgSk}. To do so, assume that  $\pvec z\not\leq \vec z$. Pick a subscript $r\in\set{1,\dots,m}$ such that $z'_r=1$ but $z_r=0$; this is the place where $\stpzf_{m+1}$  will be ``delayed'' in $\Equ{\Zfi}$.
 With  more details and computing in $\Equ{\Zfi}$, we know from  \eqref{eqBGkJggLjSjgygLsMz} that  
\begin{equation}
\stpzf_{r-1}(\stzbmu)  \leq \kequ{\stzI_j:0\leq j\leq r-1}.
\label{eqdzhBfjnwnRmRks}
\end{equation}
We know the same for $r$ instead of $r-1$ in \eqref{eqdzhBfjnwnRmRks}; however, since $\ode$ is a meetand of the term $\stpzf_{r}(\obmu)$ 
(because of \eqref{eqZgbVrSdGnNzkb} and $z'_r=1$) but $\stzgd$ fails to collapse $I_r$ since $z_r=0$, it follows that 
\begin{equation}
\stpzf_{r}(\stzbmu)  \leq \kequ{\stzI_j:0\leq j\leq r-1}.
\label{eqdzhcmsTksZr}
\end{equation}
Let, say, $r$ be odd; the even case is similar. Then, using our construction (and the dotted ovals in Figure~\ref{figZgRjmCbsWs} if $r=5$) 
at the inequality $\leq^\bullet$ below, we obtain that 
\begin{align}
\stpzf_{r+1}(\stzbmu) & \leqref{eqZgbVrSdGnNzkb} 
\bigl(\stpzf_{r}(\stzbmu)+\stzogc(\stzbmu) \bigr)\stzga
\cr
& \leqref{eqmCsltcGrna} \bigl(\stpzf_{r}(\stzbmu)+\stzgc \bigr)
\stzga \cr
& \leqref{eqdzhcmsTksZr}
\bigl(\kequ{\stzI_j:0\leq j\leq r-1}+\stzgc) \bigr)\stzga
\cr
& \kern 6pt \leq^\bullet \kequ{\stzI_j:0\leq j\leq r-1}.
\label{eqhbBkjSwTb}
\end{align}
Comparing  \eqref{eqdzhBfjnwnRmRks} and \eqref{eqhbBkjSwTb},
we can intuitively see that $\stpzf_{m+1}$ is delayed by (at least) two. More precisely, it is clear that 
if we modify the induction giving \eqref{eqBGkJggLjSjgygLsMz}
so that we use \eqref{eqhbBkjSwTb} rather than $\stpzf_{r+1}(\stzbmu)\leq  \kequ{\stzI_j:0\leq j\leq r+1}$ at the right moment, then we obtain that $\stpzf_{i+1}(\stzbmu)\leq  \kequ{\stzI_j:0\leq j\leq i-1}$ holds for $r\leq i\leq m$.  
In particular, $\stpzf_{m+1}(\stzbmu)\leq  \kequ{\stzI_j:0\leq j\leq m-1}$, which shows that $\stpzf_{m+1}$  does not get through  $\Zfi$. Thus, we have shown the validity of 
\eqref{eqpbxmCHglsbrgSk}.

Now that we are sufficiently familiar with virtual locksmithing, we turn our attention to the task of selecting appropriate configurations $\Zfi$. 
However, first we point out the following. If $\vph$ and $\pvph$ are distinct id-quadruples such 
$m_\vph=m_\pvph$, then we need a ``pin'', that is, we need a common position in $\vec z_\vph$ and $\vec z_\pvph$ that allows us to differentiate between $\Zfi$ and $\Zpfi$ by the ``key ''terms $\stzf_{m_{\vph+1}}$ and $\stpzf_{m_{\pvph+1}}$; these terms play the role of real keys  real locksmithing.
%
Also
we have to reserve additional positions in $\vec z_\vph$
for ``pins'' (that is, bits) to guarantee that $\Zfi$ can be distinguished from $\Zpfi$ by appropriate key term when $m_\pvph>m_\vph$. So we have to reserve pins (that is, positions) for two different purposes: 
to differentiate between configurations of the same Z-length, and to do this in case of different Z-lengths. These two purposes require distinct strategies, which can work simultaneously only if $m_\vph=\dim(z_\vph)$ is not too small.  
Fortunately, it will be clear from the  proof of part (B) of the theorem that 
\begin{equation}\left.
\parbox{8.9cm}{our task simplifies a lot of if we form the direct product of two distinct equivalence lattices, that is, if only two id-quadruples occur in our construction;}
\,\,\right\}
\label{eqpbxhmTnmcslFzVpjn}
\end{equation}
because then only one pin, that is, only one bit is necessary
to differentiate, and this can be done in a much easier way.

Next, we are going to present all the technical details for part (B) of the theorem.  After that part (B) has been proved, \eqref{eqpbxhmTnmcslFzVpjn} will allow us to omit the painful technicalities and give only a less formal and sketchy argument for part (A).

In order to make our toolkit applicable, we are going to redefine the direct product \eqref{eqdzGrlsTmTnhzk} so that instead of taking partition lattices of unstructured sets, we take equivalence lattices of appropriate configurations.  
For $j\in\set{1,2,\dots, d+2}$, we define 
\begin{equation}
k_j:=\zwid {m_j},\,\text{} W_j:=\set{\pair s t: 0\leq s<t<k_j};\text{ clearly, }|W_j|=w_j;
\label{eqZrmhcvTPrpLro}
\end{equation}
see \eqref{eqtxtHskzftmmCs} and \eqref{eqlGnwsqjxTr}.
For $j\in\set{1,2,\dots, d+2}$,  \eqref{eqthmmaina}--\eqref{eqthmmaind}  together with \eqref{eqsstsmsNklpdRrvsK}--\eqref{eqtrLhsjfzltgyl}  allow us to select antichains $X_j$ and $Y_j$ such that
\begin{align}
&|X_j|=p_j,\quad |Y_j|=q_j,\quad \vec x\not\leq \vec y\,\label{eqZrmhcvTPrpLra}
\\
&\kern 3cm \text{for all }j\in\set{1,\dots,d+2},\,\, \vec x\in X_j\text{, and }\,\vec y\in Y_j,
\cr
&X_1,Y_1\subseteq \DBV{d}{m_1-1},\,\, X_2,Y_2\subseteq \DBV{m_2-d}{m_2-d-1},\text{ and}
\label{eqZrmhcvTPrpLrb}
\\
&X_i,Y_i\subseteq \DBV{d+3-i} {m_i-d-1}\text{ for }i\in \set{3,4,\dots,d+2}.
\label{eqZrmhcvTPrpLrc}
\end{align}
In connection with \eqref{eqZrmhcvTPrpLra}, note that 
$\vec x\not\leq \vec y$ is clear by \eqref{eqxzbtSWslSGqlRdkLLvN}  and \eqref{eqtrLhsjfzltgyl}  if we go, say, after \eqref{eqthmmaind}. The inequality $\vec x\not\leq \vec y$ is also clear if we go after, say, 
\eqref{eqthmmainc}, because then we simply partition an antichain of size $p_i+q_i$ into a disjoint union $X_i\cup Y_i$; see \eqref{eqsstsmsNklpdRrvsK}. We also need the following antichains:
\begin{align}
& X_1^+:=\set{0\lev \vec x: \vec x\in X_1},\qquad\qquad Y_1^+:=\set{0\lev \vec y: \vec y\in Y_1},
\label{eqZhgBhvgRtkRpRa}
\\
&   X_j^+:=\set{1^{j-1}\lev 0^{d+2-j}\lev \vec x: \vec x\in X_j}
\,\text{ for }j\in\set{2,\dots,d+2},
\label{eqZhgBhvgRtkRpRb}
\\
&  Y_j^+:=\set{1^{j-1}\lev 0^{d+2-j}\lev \vec y: \vec y\in Y_j}
\,\,\,\text{ for }j\in\set{2,\dots,d+2}.
\label{eqZhgBhvgRtkRpRc}
\end{align}
Associated with our antichains, we define some sets of id-quadruples as follows. For $j\in\set{1,2,\dots, d+2}$, let
\begin{align}
& \Psi_j:=\set{\tuple{m_j,1,1,\vec z}: \vec z\in X_j^+},\quad\Gamma_j:=\set{\tuple{m_j,s,t,\vec z}: \vec z \in Y_j^+ ,\,\, \pair s t\in W_j},
\label{eqztspRstlFtdPlTcrp}
\\
& \Phi_j:= \Psi_j\cup \Gamma_j,\quad\text{ and }\quad\Phi:=\bigcup_{i=1}^{d+2}\Phi_i . \label{eqhpfpflvdjknenvKnwwm} 
\end{align}
Note at this point that, for $j\in\set{2,3,\dots, d+2}$, every  $\vec z\in X_j^+\cup  Y_j^+$ is of dimension
$(j-1)+(d+2-j)+ (m_j-d-1)=m_j$ while the vectors in $ X_1^+\cup  Y_1^+$ are of dimension $m_1$. Hence, $\Phi$ is indeed a set of id-quadruples.
Armed with $\Phi$ defined in \eqref{eqhpfpflvdjknenvKnwwm}, we define a lattice $L$ as follows.
\begin{equation}
L:=\prod_{\vph\in \Phi}\Equ\Zfi.
\label{eqmslgzRnprmrFzrmhSdmSmdR}
\end{equation}
We went after  \eqref{eqlGnwsqjxTr}-- \eqref{eqthmmaind} when defining
the $W_j$ in \eqref{eqZrmhcvTPrpLro}, the antichains \eqref{eqZrmhcvTPrpLrb}--\eqref{eqZhgBhvgRtkRpRc}, and the sets of id-quadruples in \eqref{eqztspRstlFtdPlTcrp}--\eqref{eqhpfpflvdjknenvKnwwm}. Therefore, 
$|\Psi_j|=p_j$ and $|\Gamma_j|=q_jw_j$ for all $j\in\set{1,2,\dots, d+2}$, and we conclude that 
the direct product \eqref{eqdzGrlsTmTnhzk} and $L$ defined in \eqref{eqmslgzRnprmrFzrmhSdmSmdR} are isomorphic lattices.  Thus, it suffices to prove that $L$ is four-generated. 
The elements of $L$ are $|\Phi|$-dimensional vectors; equivalently, they are 
choice functions $\vec\tau\colon  \Phi \to \bigcup_{\vph\in \Phi}\Equ\Zfi$, that is, functions with the property $\vec\tau(\vph)\in \Equ\Zfi$ for all $\vph\in \Phi$. However, we will often use the visual notation $\tuple{\vec\tau(\vph): \vph\in\Phi}$ instead of $\vec\tau$. 
With the intention of presenting the elements of a four-element generating set, we define $\valpha,\vbeta,\vgamma,\vdelta\in L$ by defining their components in the direct factors in the following natural way:
\begin{equation}
\begin{aligned}
\valpha&=\tuple{\stzga:\vph\in\Phi}, \quad&
\vbeta&=\tuple{\stzgb:\vph\in\Phi},\cr
\vgamma&=\tuple{\stzgc:\vph\in\Phi}, \quad&
\vdelta&=\tuple{\stzgd:\vph\in\Phi}.
\end{aligned}
\label{eqnkvVkrmKmpkWzslnk}
\end{equation}
We claim that 
\begin{equation}\left.
\parbox{8.9cm}{for any $\vph,\pvph\in\Phi$ and $0\leq i\leq m_\pvph+1$, the term $\stpzf_i$ gets through $\Zfi$ if and only if either $\vph=\pvph$ and $i=m_\vph+1$, or   there is a (unique) $j\in\set{1,\dots, d+2}$ such that $\vph\in \Psi_j$, $\pvph\in\Gamma_j$, $i=m_\vph+1=m_\pvph+1$, and $\vec z_\pvph\leq \vec z_\vph$.}
\,\,\,\right\}
\label{eqpbxZhdhBvTdnbf}
\end{equation}
Remember from  \eqref{eqtxthnGtvGcslfrlTdLmb} that 
 $0\leq i\leq m_\pvph+1$ means that $\stpzf_i$ is defined.
The ``if part'' of \eqref{eqpbxZhdhBvTdnbf} follows from \eqref{eqpbxmCHglsbrgSk}. In order to show the ``only if part'', assume that  $\stpzf_i$ gets through $\Zfi$. It follows from \eqref{eqtxtZghthnGwpndJCn} that  $i \geq m_\vph+1$. Since $m_\pvph+1\geq i$, we obtain that $m_\pvph\geq  m_\vph$. 

As the next step towards \eqref{eqpbxZhdhBvTdnbf}, we are going to show that 
\begin{equation}\left.
\parbox{7.5cm}{if $\vph,\pvph\in\Phi$, $0\leq i\leq m_\pvph+1$, and $m_\pvph >  m_\vph$, then $\stpzf_i$ does not get through $\Zfi$.}
\,\,\,\right\}
\label{eqpbwprnPfrLndsStmLCr}
\end{equation}
%
Since $m_\pvph >  m_\vph$, it follows from the first sentence of part (B) of Theorem~\ref{thmmain} (which is only an
 assumption here but not the statement we are proving) and \eqref{eqztspRstlFtdPlTcrp}--\eqref{eqhpfpflvdjknenvKnwwm} that $\vph\in \Phi_j$ and $\pvph\in \Phi_{j'}$ with $1\leq j<j'\leq d+2$. First, if $j=1$, then 
the first bit of $\vec z_\vph$ and that of $\vec z_\pvph$  are 0 and 1, respectively, by \eqref{eqZhgBhvgRtkRpRa}--\eqref{eqztspRstlFtdPlTcrp}.
Hence, \eqref{eqpbxBsmsgzsghWq} yields that  $\stpzf_i$ does not get through $\Zfi$, as required. 

Second, assume that $2=j<j'\leq d+2$. We know from  \eqref{eqZrmhcvTPrpLrb} and \eqref{eqZhgBhvgRtkRpRb}--\eqref{eqhpfpflvdjknenvKnwwm} that  $\vec z$ is of the form
$\vec z=1\lev 0^d \lev \vec x$ with $\vec x\in\DBV{m_2-d}{m_2-d-1}$; see \eqref{eqZrmhcvTPrpLrb}. Since $m_1$ and $m_2$ are \emph{odd} numbers and $d\leq m_1<m_2$, we obtain that both parameters of $\DBV{m_2-d}{m_2-d-1}$ are positive; this is necessary to make the definition of this set in 
 \eqref{eqpbxZFUdmszmznflkKbL} meaningful.  (Note that since $m_2-d>m_2-d-1$, the first parameter tailors no restriction on the airiness.) Hence, $\vdim{\vec x}>0$ and the first bit of $\vec x$ is 1. 
Since $2<j'$,  \eqref{eqZrmhcvTPrpLrc} and \eqref{eqZhgBhvgRtkRpRb}--\eqref{eqZhgBhvgRtkRpRc} tell us that $\pvec z$ is of the form  $\pvec z=1^{j'-1}\lev 0^{d+2-j'}\pvec x$ 
with  $\pvec x\in\DBV{d+3-j'}{m_{j'}-d-1}$. Since $j'\leq d+2$ and $m_{j'}-d-1 >m_2-d-1>0$, the parameters of $\DBV{d+3-j'}{m_{j'}-d-1}$ are positive.
Hence,  \eqref{eqpbxZFUdmszmznflkKbL} yields that $\air{\pvec x}< d+3-j'\leq d$ and that the first bit of $\pvec x$ is 1. Since the zeros in $\pvec z$ are separated by the first bit of $\pvec x$, we obtain that 
\begin{equation}
\begin{aligned}
\air{\pvec z}&=\max\Bigl(\air{1^{j'-1}\lev 0^{d+2-j'}}, \air{\vec x}  \Bigr)\cr
&\leq \max(d+2-j', d-1)\leq d-1.
\end{aligned}
\label{eqczhGnrTshmnFgnhW}
\end{equation}
Now, we are in the position to conclude from  $\vec z=1\lev 0^d \lev \vec x$, \eqref{eqpbxBlRSblJsNg}, and  \eqref{eqczhGnrTshmnFgnhW} that   $\stpzf_i$ does not get through $\Zfi$, as required.

Third, we assume that $3\leq j<j'\leq d+2$. Using \eqref{eqZrmhcvTPrpLrc} and \eqref{eqZhgBhvgRtkRpRb}--\eqref{eqZhgBhvgRtkRpRc} as above, we obtain that 
$\vec z$ and  $\pvec z$ are of the forms  
$\vec z=1^{j-1}\lev 0^{d+2-j}\vec x$ 
with  $\vec x$ in $\DBV{d+3-j}{m_{j}-d-1}$ and 
$\pvec z=1^{j'-1}\lev 0^{d+2-j'}\pvec x$ 
with  $\pvec x$ belonging to $\DBV{d+3-j'}{m_{j'}-d-1}$. 
As earlier, the parameters of these two sets are positive. 
Using \eqref{eqpbxZFUdmszmznflkKbL}, we obtain  that  $\air{\pvec x}< d+3-j'$ and the first bit of $\pvec x$ is 1. This bit serves as a separator again and, similarly to \eqref{eqczhGnrTshmnFgnhW},  we obtain that $\air{\pvec z}\leq d+2-j'$. But $d+2-j>d+2-j'\geq 0$, whereby 
 $\vec z=1^{j-1}\lev 0^{d+2-j}\vec x$,   $\air{\pvec z}\leq d+2-j'< d+2-j$ and \eqref{eqpbxBlRSblJsNg} imply that $\stpzf_i$ does not get through $\Zfi$, as required.  I this way, \eqref{eqpbwprnPfrLndsStmLCr} has been proved.

Armed with  \eqref{eqpbwprnPfrLndsStmLCr}, we resume the argument for the ``only if part'' of \eqref{eqpbxZhdhBvTdnbf}. So $\stpzf_i$ is assumed to get through $\Zfi$. 
If $\vph=\pvph$, then it follows from  $i\leq m_\pvph+1$ and   \eqref{eqtxtZghthnGwpndJCn} that $i = m_\vph+1$, as required. Hence, we can assume that $\vph\neq\pvph$. 
We showed right before  \eqref{eqpbwprnPfrLndsStmLCr} that $m_\pvph\geq m_\vph$. This inequality and 
 \eqref{eqpbwprnPfrLndsStmLCr} yield that  $m_\pvph = m_\vph$.  
Since $i\leq m_\pvph +1= m_\vph+1$ by \eqref{eqtxthnGtvGcslfrlTdLmb}, we obtain from \eqref{eqtxtZghthnGwpndJCn} that 
$i= m_\pvph +1= m_\vph+1$, as required. So it is
$\stpzf_{m+1}$ what gets through $\Zfi$, whereby \eqref{eqpbxmCHglsbrgSk} implies that 
\begin{equation}
\vec z_\pvph\leq \vec z_\vph.
\label{eqhTmfkPsTmFJcpKhWn}
\end{equation}
Also,  $m_\pvph = m_\vph$ implies that there is a unique $j\in\set{1,\dots, d+2}$ such that both $\vph$ and $\pvph$ belong to $\Phi_j$. Using that  $X_j$ and $Y_j$ in \eqref{eqZrmhcvTPrpLra} are antichains, it follows from $\vph\neq\pvph$, \eqref{eqZrmhcvTPrpLra}, \eqref{eqztspRstlFtdPlTcrp},  \eqref{eqhpfpflvdjknenvKnwwm}, and the inequality 
 $\vec z_\pvph\leq \vec z_\vph$  that $\vph\in \Psi_j$ and $\pvph\in\Gamma_j$. This completes the proof of \eqref{eqpbxZhdhBvTdnbf}. 

Next, for later use, we show that whenever $x$, $y$, $z$, and $w$ are elements of a set $A$ such that  $x\neq y$ and $z\neq w$, then
\begin{equation}
 \equ x y\cdot(\equ x z+\equ w y)\cdot(\equ x w+\equ z y)=\enul
\label{equchnGyksGncPtlWmcs}
\end{equation}
holds in $\Equ A$. Since $x$ and $y$ play a symmetric role and so do $z$ and $w$, there are only three cases to deal with: $\set{x,y,z,w}|=4$, $\set{x,y,z,w}|=3$ with $y=z$, and $\set{x,y,z,w}|=2$ with $\pair x y=\pair z w$. In each of these three cases,
\eqref{equchnGyksGncPtlWmcs} trivially holds. This shows the validity of \eqref{equchnGyksGncPtlWmcs}. 

Next, with reference to \eqref{eqMssnzLlBkbsZwhgmLrRF}, we define the following auxiliary term.
\begin{equation}
\stzdotsg(\obmu):=\stzeuv {a_s}{a_{t+1}}(\obmu)\cdot(\stzeuv {a_s}{c}(\obmu)+\stzeuv {a_{t+1}}{c}(\obmu)),\text{ where }s = s_\vph\text{ and }t = t_\vph.
\label{eqkplPndjTlrsDbkt}
\end{equation}
It follows from \eqref{eqzkfnbbBNhJdMr} that, 
regardless whether $s=s_\vph$ is less than $t=t_\vph$ or not, 
\begin{equation}
\stzdotsg(\stzbmu)=\equ{\stza_s}{\stza_{t+1}} \text{ holds in }
\Equ\Zfi.
\label{eqkpTmmSfrndptKG}
\end{equation}
With $s:=s_\vph$, $t:=t_\vph$, and $k:=k_\vph$, we define a term $\stzkg$ as follows. Let
\begin{equation}
\begin{aligned}
\stzkg(\obmu):= 
\stzf_{m+1}(\obmu)\cdot \rftrm 0(\obmu) &\cdot (\stzeuv {a_{k-1}}{a_s}(\obmu)+\stzdotsg(\obmu)+ 
\stzeuv {a_{t+1}}{a_k}(\obmu))    \cr
&\cdot   (\stzeuv {a_{k-1}}{a_{t+1}}(\obmu)+\stzdotsg(\obmu)+\stzeuv {a_{s}}{a_k}(\obmu)).
\end{aligned}
\label{eqpHltstszb}
\end{equation}
Note that for any other id-quadruple $\pvph$, the terms  $\stpzdotsg(\obmu)$ and $\stpzkg(\obmu)$ are also defined by 
\eqref{eqkplPndjTlrsDbkt} and \eqref{eqpHltstszb}
 but then, of course, $s$, $t$, and $m$ in \eqref{eqkplPndjTlrsDbkt} and  \eqref{eqpHltstszb} stand for $s_\pvph$, $t_\pvph$, and $m_\pvph$, respectively. We claim that for any $\vph,\pvph\in \Phi$,
\begin{equation}
\stpzkg(\stzbmu)=\begin{cases}
\kequ{\stzI_{m_\vph+1}},&\text{ if }\pvph=\vph;\cr
\enul,&\text{ if }\pvph \neq \vph.
\end{cases}
\label{eqzhGszpfCsLjpHlG}
\end{equation}
In order to prove \eqref{eqzhGszpfCsLjpHlG}, first we assume  that $\pvph=\vph$. 
It follows from 
\eqref{eqmmMcnVgSkrgBrsbkMS} and  \eqref{eqinTrwkhFt} 
that the meet of the first two meetands of 
\eqref{eqpHltstszb} becomes $\kequ{\stzI_{m_\vph+1}}$ after substituting $\stzbmu$ for $\obmu$, that is, 
$\stzf_{m_\vph+1}(\stzbmu)\cdot \rftrm 0(\stzbmu)= \kequ{\stzI_{m_\vph+1}} =\equ{\stza_{k_\vph-1}}{\stza_{k_\vph}}$. By \eqref{eqzkfnbbBNhJdMr} and \eqref{eqkpTmmSfrndptKG},   the same substitution turns the third meetand, which is a three-fold join, into 
$\equ {\stza_{k_\vph-1}}{\stza_{s_\vph}}  + \equ {\stza_{s_\vph}}{\stza_{t_\vph+1}}  + 
\equ {a_{t_\vph+1}}{a_{k_\vph}}$, which is greater than 
$\equ{\stza_{k_\vph-1}}{\stza_{k_\vph}}=\kequ{\stzI_{m_\vph+1}}$. 
We obtain similarly that the fourth meetand becomes greater than 
$\kequ{\stzI_{m_\vph+1}}$ after substituting $\stzbmu$ for $\obmu$. Hence, $\stpzkg(\stzbmu)=\kequ{\stzI_{m_\vph+1}}$ if $\pvph=\vph$, as required.

Second, assume that  $\pvph\neq\vph$. We can assume that there is a (unique) $j\in\set{1,\dots, d+2}$ such that 
\begin{equation}
\vph\in \Psi_j, \,\, \pvph\in\Gamma_j,\,\,  m_\vph+1=m_\pvph+1=m_j+1,\,\,\text{and}\,\,\vec z_\pvph\leq \vec z_\vph,
\end{equation}
since otherwise   \eqref{eqinTrwkhFt},  
 \eqref{eqpbxZhdhBvTdnbf}, and the first two meetands in \eqref{eqpHltstszb} immediately imply that 
$\stpzkg(\stzbmu)=\enul$ holds in $\Equ\Zfi$. Since  $\vph\in \Psi_j$ and $\pvph\in \Gamma_j$, we have  that 
the necktie of $\Zfi$ is trivial but $s':=s_\pvph<t_\pvph <t_\pvph+1=:t'+1$. Hence, it follows from the construction of $\Zfi$ and  $(t'+1)-s'\geq 2$ that 
\begin{equation}
\stzgb\cdot(\stzgc + \equ{\stza_{s'}}{\stza_{t'+1}})=\enul\,\,
\text{ and }\,\stzgc\cdot(\stzgb + \equ{\stza_{s'}}{\stza_{t'+1}})=\enul.
\label{eqzHfGskgHzwPWRt}
\end{equation}
Combining \eqref{eqszGnslmsLmjTzZbh} and \eqref{eqzHfGskgHzwPWRt},
we obtain that
\begin{equation*}
\stzgb\cdot(\stzgc + \stpzeuv{a_{s'}}{a_{t'+1}}(\stzbmu))=\enul\,\,
\text{ and }\,\stzgc\cdot(\stzgb + \stpzeuv{a_{s'}}{a_{t'+1}}(\stzbmu))=\enul.
\end{equation*}
After comparing these two equalities with \eqref{eqzghbjRwJF} (applied to $\pvph$ rather than to $\vph$), 
\begin{equation}
\stpzeuv{a_{s'}}{c}(\stzbmu)=\enul\,\,\text{ and }\,\,\stpzeuv{a_{t'+1}}{c}(\stzbmu)=\enul. 
\label{eqczshrPbTNXnHqxj}
\end{equation}
With $m:=m_\vph=m_\pvph$ and $k:=k_\vph=k_\pvph$,
we have that
\begin{equation}
\stpzf_{m+1}(\stzbmu)\cdot \rftrm 0(\stzbmu)\leq  \rftrm 0(\stzbmu)
\overset{\eqref{eqinTrwkhFt}}\leq \kequ{\stzI_{m+1}}=\equ{\stza_{k-1}}{\stza_k}.
\label{eqdzhGhsKtkmRdKQRvLn}
\end{equation}
The equalities in \eqref{eqczshrPbTNXnHqxj} show that both joinands in \eqref{eqkplPndjTlrsDbkt}, applied to $\pvph$ rather than to $\vph$, 
turn to $\enul$ if $\stzbmu$ is substituted for $\obmu$. Hence,   $\stpzdotsg(\stzbmu)=\enul$ in $\Equ\Zfi$. Thus, the middle joinands disappear in \eqref{eqpHltstszb}, applied to $\pvph$ rather than to $\vph$, after letting $\obmu:=\stzbmu$.
This fact together with \eqref{eqdzhGhsKtkmRdKQRvLn} show that,
with the notation $s':=s_\pvph$ and $t':=t_\pvph$,
\begin{align*}
\stpzkg(\stzbmu)\leq 
\equ{\stza_{k-1}}{\stza_k}\cdot \bigl(\stpzeuv {a_{k-1}}{a_{s'}}(\stzbmu)+\stpzeuv {a_{t'+1}}{a_k}(\stzbmu)\bigr)  \cr  
\cdot   \bigl(\stpzeuv {a_{k-1}}{a_{t'+1}}(\stzbmu)+\stpzeuv {a_{s'}}{a_k}(\stzbmu)\bigr)\cr
\leqref{eqszGnslmsLmjTzZbh} \equ{\stza_{k-1}}{\stza_k}\cdot
\bigl(\equ{\stza_{k-1}}{\stza_{s'}} + \equ{\stza_{t'+1}}{\stza_{k}}    \bigr) \cr 
\cdot 
\bigl(\equ{\stza_{k-1}}{\stza_{t'+1}} + \equ{\stza_{s'}}{\stza_{k}}    \bigr) \eeqref{equchnGyksGncPtlWmcs} \enul.
\end{align*}
This completes the proof of \eqref{eqzhGszpfCsLjpHlG}.

We still have to define some terms. For $\pvph\in \Phi$, let
\allowdisplaybreaks{
\begin{align}
&\stpzbg(\obmu):= \bigl(\stpzkg(\obmu)+\oga\ode\bigr)\cdot \stpzogb(\obmu),\,\,\text{ see \eqref{eqmCsltcGrna} and \eqref{eqpHltstszb}},
\label{eqmTrkfjzSgkhLtksRlca}\\
&\stpzogtwb(\obmu):=  \stpzogb(\obmu)\cdot\bigl({}\stpzbg(\obmu) + \oal  \bigr),\label{eqmTrkfjzSgkhLtksRlcb}\\
&\stpzogtwc(\obmu):=  \stpzogc(\obmu)\cdot\bigl({}\stpzbg(\obmu) + \oal  \bigr),\label{eqmTrkfjzSgkhLtksRlcc}\\
&\stpzogtwa(\obmu):= \oal\cdot\bigl(\, \stpzogtwb(\obmu) + \stpzogtwc(\obmu)   \bigr),\label{eqmTrkfjzSgkhLtksRlcd}\\
&\stpzogthb(\obmu):=\obe\cdot\bigl(\, \stpzogtwa(\obmu) + \oga  \bigr), \label{eqmTrkfjzSgkhLtksRlcf}\\
&\stpzogthc(\obmu):=\oga\cdot\bigl(\, \stpzogtwa(\obmu) + \obe  \bigr),\,\,\text{ and}\label{eqmTrkfjzSgkhLtksRlcg}\\
&\stpzogthd(\obmu):=\ode\cdot\bigl(\, \stpzogthb(\obmu) + \stpzogthc(\obmu)  \bigr).\label{eqmTrkfjzSgkhLtksRlck}
\end{align}}%
Our goal with \eqref{eqmTrkfjzSgkhLtksRlca}--\eqref{eqmTrkfjzSgkhLtksRlck} is revealed by the following equalities; each of them will follow from earlier equalities out of \eqref{eqmTrkfjzSgkhLtksRlca}--\eqref{eqmTrkfjzSgkhLtksRlck}, from the structure of $\Zfi$, and from some other equalities that will be indicated. To prepare some references to \eqref{eqmgszRtmTBvsGa} and \eqref{eqmgszRtmTBvsGb}, note that the terms occurring in them do not depend on $\pvph$; see \eqref{eqmCsltcGrna}.
For any two $\vph,\pvph\in \Phi$, we have that 
\allowdisplaybreaks{
\begin{align}
&\stpzbg(\stzbmu)=\begin{cases}
\equ{\stza_{k_\vph-1}}{\stzb_{k_\vph-1}},&\text{ if }\pvph=\vph,\cr
\enul,&\text{ if }\pvph\neq \vph ,
\end{cases}\,\text{ by  \eqref{eqmgszRtmTBvsGa} and \eqref{eqzhGszpfCsLjpHlG};}
\label{alignSmpgtKPsRha}\\
&\stpzogtwb(\stzbmu)=\begin{cases}
 \stzogb(\stzbmu),&\text{ if }\pvph=\vph,\cr
\enul,&\text{ if }\pvph\neq \vph ,
\end{cases}\,\,\text{ by \eqref{eqmgszRtmTBvsGa} and \eqref{alignSmpgtKPsRha};}
\label{alignSmpgtKPsRhb}\\
&\stpzogtwc(\stzbmu)=\begin{cases}
 \stzogc(\stzbmu),&\text{ if }\pvph=\vph,\cr
\enul,&\text{ if }\pvph\neq \vph ,
\end{cases}\,\,\text{ by \eqref{eqmgszRtmTBvsGb} and \eqref{alignSmpgtKPsRha};}
\label{alignSmpgtKPsRhc}\\
&\stpzogtwa(\stzbmu)=\begin{cases}
 \stzga,&\text{ if }\pvph=\vph,\cr
\enul,&\text{ if }\pvph\neq \vph ,
\end{cases}\,\,\text{ by \eqref{eqmgszRtmTBvsGa}--\eqref{eqmgszRtmTBvsGb} and   \eqref{alignSmpgtKPsRhb}--\eqref{alignSmpgtKPsRhc};}
\label{alignSmpgtKPsRhd}\\
&\stpzogthb(\stzbmu)=\begin{cases}
 \stzgb,&\text{ if }\pvph=\vph,\cr
\enul,&\text{ if }\pvph\neq \vph ,
\end{cases}\,\,\text{ by \eqref{alignSmpgtKPsRhd};}
\label{alignSmpgtKPsRhe}\\
&\stpzogthc(\stzbmu)=\begin{cases}
 \stzgc,&\text{ if }\pvph=\vph,\cr
\enul,&\text{ if }\pvph\neq \vph ,
\end{cases}\,\,\text{ by \eqref{alignSmpgtKPsRhd};}
\label{alignSmpgtKPsRhf}\\
&\stpzogthd(\stzbmu)=\begin{cases}
 \stzgd,&\text{ if }\pvph=\vph,\cr
\enul,&\text{ if }\pvph\neq \vph 
\end{cases}\,\,\text{ by \eqref{alignSmpgtKPsRhe} and \eqref{alignSmpgtKPsRhf}.}
\label{alignSmpgtKPsRhg}
\end{align}}%
Finally, if we apply the quaternary term $\stpzogtwa(\stzbmu)$ from \eqref{eqmTrkfjzSgkhLtksRlcd} to the quadruple $\tuple{\valpha,\vbeta,\vgamma,\vdelta}$ defined in \eqref{eqnkvVkrmKmpkWzslnk}, then we do this componentwise. Hence,
it follows from \eqref{alignSmpgtKPsRhd} that 
\begin{equation}
\stpzogtwa(\valpha,\vbeta,\vgamma,\vdelta)=
\tuple{\enul,\dots,\enul, \stpzga ,\enul,\dots,\enul},
\label{eqmRmVsSzsTtfNxkdKna}
\end{equation}
where $\stpzga\in \Zpfi$ is in the $\pvph$-th component of $L$ defined  in \eqref{eqmslgzRnprmrFzrmhSdmSmdR}. Using \eqref{alignSmpgtKPsRhe}--\eqref{alignSmpgtKPsRhg}, we obtain similarly that 
\begin{align}
\stpzogthb(\valpha,\vbeta,\vgamma,\vdelta)=
\tuple{\enul,\dots,\enul, \stpzgb ,\enul,\dots,\enul}
\label{eqmRmVsSzsTtfNxkdKnb}\\
\stpzogthc(\valpha,\vbeta,\vgamma,\vdelta)=
\tuple{\enul,\dots,\enul, \stpzgc ,\enul,\dots,\enul}
\label{eqmRmVsSzsTtfNxkdKnc}\\
\stpzogthd(\valpha,\vbeta,\vgamma,\vdelta)=
\tuple{\enul,\dots,\enul, \stpzgd ,\enul,\dots,\enul}.
\label{eqmRmVsSzsTtfNxkdKnd}
\end{align}
Let $S:=\sublat{\valpha,\vbeta,\vgamma,\vdelta}$, the sublattice of $L$ generated by $\set{\valpha,\vbeta,\vgamma,\vdelta}$. The elements of $L$ described in \eqref{eqmRmVsSzsTtfNxkdKna}-- \eqref{eqmRmVsSzsTtfNxkdKnd} belong to $S$ since they are obtained from the generators of $S$ with the help of lattice terms. It follows from 
Lemma~\ref{lemmaZoddxzG} that 
\begin{equation}
\text{for every }\stpzge\in \Equ\Zpfi,\quad
\tuple{\enul,\dots,\enul, \stpzge ,\enul,\dots,\enul}\in S,
\label{eqmkSrLnSsrKdWrcPctzl}
\end{equation}
where $\stpzge$ is the $\pvph$-th component. Since each element of $L$ is the join of elements of form \eqref{eqmkSrLnSsrKdWrcPctzl}, $L=S$. Hence, $L$ is four-generated, 
and  part (B) of the theorem has been proved.

Next, we turn our attention to part (A). There are four cases to deal with.

\begin{case}\label{caseone} We assume that 
 $7\leq n=n'-1$ and $n$ is odd. Let $d:=\len n$; see \eqref{eqkrSnmkGtmKlPtrZg}, and let $m_1:=d$. By  \eqref{eqkrSnmkGtmKlPtrZg}, $m_1=\len {n'}$, $n=m_1+4$; and $n'=m_1+5$; compare the last two equalities to \eqref{eqhCshbfhRvhswkm}. Since $n\geq 7$, we have that $d:=m_1\geq 3$. Hence, $\pair 1 1\in \tra d{d-1} = \tra{d}{m_1-1}$ either by \eqref{eqcztsmsTrhltmRb} and \eqref{eqpbxmNkSgtSnkMhG}, if $d>3$, or by   \eqref{eqcskknwgmLmnxFnT}, if $d=3$. 
Thus, we can let $p_1=q_1=1$ and, for $i\in\set{2,\dots, d+2}$, $p_i=q_i=0$ according to  \eqref{eqpbxmNkSgtSnkMhG},  \eqref{eqthmmaina}, \eqref{eqthmmainb}, and \eqref{eqthmmaind}. 
By Remark~\ref{remarkcZgBnjT}, we can disregard the stipulation $p_i+q_i>0$. Similarly, instead of going after \eqref{eqlGnwsqjxTr},
Remark~\ref{remarkcZgBnjT} allows us to take $w_1:=1$. Then the lattice given in \eqref{eqdzGrlsTmTnhzk} is clearly (isomorphic to) $\Part n\times \Part {n'}$, whereby this direct product is four-generated by part (B) of the theorem, as required.  
\end{case}

\begin{case}\label{casetwo}  We assume that either
 $7\leq n<n'-1$ and $n$ is odd, or $7<n<n'$ and $n$ is even.
By  \eqref{eqkrSnmkGtmKlPtrZg}, $m_1:=\len n$ is smaller than $m_2:=\len{n'}$ and $d:=m_1$ is at least 3. As in Case~\ref{caseone}, $\pair 1 1$ belongs to $\tra{d}{m_1-1}$. By   \eqref{eqpbxmNkSgtSnkMhG}, so do $\pair 1 0$ and $\pair 0 1$.  So we can let $\pair{p_1}{q_1}:=\pair 1 0$ if $n$ is odd, and we can  let  $\pair{p_1}{q_1}:=\pair 01$ if $n$ is even; 
this choice complies with \eqref{eqthmmaina}. Similarly, let $\pair{p_2}{q_2}:=\pair 1 0$ if $n'$ is odd and 
 $\pair{p_2}{q_2}:=\pair 01$ if $n'$ is even. 
Since $m_2>m_1=d$ and both numbers are odd, $m_2-d-1$ from \eqref{eqthmmainb} is positive. Hence, it follows from 
\eqref{eqhmghZtkkrLnspSg} and \eqref{eqnHspjzlkHMkLbR} that the choice of  $\pair{p_2}{q_2}$ complies with \eqref{eqthmmainb}. 
 Choosing the rest of $p_i$'s and $q_i$'s to be 0 and all the $w_i$'s to be 1, as  in Case~\ref{caseone}, 
 the lattice given in \eqref{eqdzGrlsTmTnhzk} turns into $\Part n\times \Part {n'}$, and part (B) completes this case as previously. 
\end{case}

Since all possibilities with $7\leq n$ have been settled, 
the rest of the cases will assume that  $n\in\set{5,6}$. 

\begin{case} We assume that $n\in\set{5,6}$ and $n'\in\set{5,6}$. This means that $n=5$ and $n'=6$, because $n<n'$. 
While dealing with this case, $\vec 0$ and $\vec 1$ will denote the two 1-dimensional bit vectors. Let $\vph=\tuple{1,1,1,\vec 1}$ and $\pvph=\tuple{1,0,1,\vec 0}$. Since $|\Zfi|=n_\vph=5$ and $|\Zpfi|=n_\pvph=6$, it suffices to show that $\Equ\Zfi\times\Equ\Zpfi$ is four-generated. Using \eqref{eqpbxmCHglsbrgSk} and \eqref{eqpbxhmTnmcslFzVpjn}, a simplified version of the technique used for part (B) yields in a straightforward way that $\Equ\Zfi\times\Equ\Zpfi$ is indeed four-generated; the tedious details are omitted.
\end{case}

\begin{case} We assume that $n\in\set{5,6}$ but $n'\notin\set{5,6}$. Then $m:=\len n=1$,  $n'\geq 7$, and  $m':=\len{n'}\geq 3$. It is important that $m<m'$.  We let $\vph=\tuple{1,s,t,\vec 0}$ where $\vec 0$ is 1-dimensional, 
$\pair s t=\pair 1 1$ if $n=5$, and  $\pair s t=\pair 0 1$ if $n=6$. With the $m'$-dimensional ``unit vector'' $\vec 1=1^{m'}=\tuple{1,\dots,1}$, we let $\pvph=\tuple{1,1,1,\vec 1}$ if $n'$ is odd while we let  $\pvph=\tuple{1,0,1,\vec 1}$ if $n'$ is even. Based on \eqref{eqtxtZghthnGwpndJCn}, $m<m'$,  
 and   \eqref{eqpbxBsmsgzsghWq},
a simplified version of the technique used for part (B) yields that  $\Equ\Zfi\times\Equ\Zpfi = \Equ n\times \Equ{n'}$ is four-generated; the straightforward tedious details are omitted again. 
This completes the proof of Theorem~\ref{thmmain}. \qedhere
\end{case}
\end{proof}

\begin{proof}[Proof of Corollary~\ref{coroltconsecutive}]
Keeping the notation of Theorem~\ref{thmmain} in effect, assume that $d\geq 3$. For $i\in\set{1,2,\dots, d+1}$, we let 
$m_i=d+2i$.  (Note that $m_{d+2}$ is not used here.)
Since $\tra d{m_1-1}=\tra d{d+1}$ contains $\pair 1 1$ by 
\eqref{eqcztsmsTrhltmRb} and \eqref{eqpbxmNkSgtSnkMhG},
\eqref{eqthmmaina} allows us to chose $\pair{p_1}{q_1}=\pair 1 1$. Similarly, $\pair 1 1\in \tra{m_2-d}{m_2-d-1}=\tra 4 3$ by \eqref{eqcztsmsTrhltmRb} and \eqref{eqpbxmNkSgtSnkMhG}, so \eqref{eqthmmainb} allows us to let $\pair{p_2}{q_2}:=\pair 1 1$.  
For $i\in\set{3,4,\dots, d+1}$,
$\tra{d+3-i}{m_i-d-1}=\tra{d+3-i}{2i-1}$, which contains $\pair 1 1$ since $d+3-i\geq 2$ and we can apply  \eqref{eqcztsmsTrhltmRb} and  \eqref{eqpbxmNkSgtSnkMhG}. Hence, \eqref{eqthmmaind} complies with $\pair{p_i}{q_i}=\pair 1 1$.  So if we reduce $w_i$ down to 1, which is possible by Remark~\ref{remarkcZgBnjT}, then 
\begin{equation}\left.
\parbox{9cm}{$\Part{d+6} \times \Part{d+7} \times \Part{d+8} \times
\dots\times \Part{3d+7}$ is four-generated}
\,\,\,\right\}
\label{eqcmzGmhKmpgnccwfLs}
\end{equation} 
by Theorem~\ref{thmmain} since
\begin{equation}
\parbox{10cm}{$n_1'=m_1+4=d+6$, $n_1''=m_1+5=d+7$, $n_2'=m_2+4=d+2\cdot 2 +4=d+8$, \dots, $n_{d+1}''=m_{d+1}+5=d+2(d+1)+5=3d+7$.}
\label{eqpbxmtlKhvmhswbsszzrsk}
\end{equation}
Now there are two cases. First, if $n\geq 9$ is odd, then we can let $n_1':=n$. This defines $d$ as $n-6$ by \eqref{eqpbxmtlKhvmhswbsszzrsk}, and $3d+7$ equals $3(n-6)+7=3n-11\geq 3n-14$. Second, if $n\geq 9$ is even, then we  let $n_1'':=n$. Then $d=n-7$ by \eqref{eqpbxmtlKhvmhswbsszzrsk}, which leads to 
$3d+7=3n-14$.  In both cases, the
 corollary follows  from 
\eqref{eqcmzGmhKmpgnccwfLs}  and Remark~\ref{remarkcZgBnjT}.
\end{proof}

\begin{proof}[Proof of Corollary~\ref{corolmanyfactors}]
We can assume that $u$ is a (large) even number, so $u=2v$ for, say, $9\leq v\in\NN$. In harmony with the notation used in Theorem~\ref{thmmain}, let $d:=4v+1$ and, for $i\in\set{1,2,\dots,v}$,
let $m_i:=8v+2i-1$. Similarly to the previous proof, $m_i$ for $i$ outside the chosen index set will be disregarded and $w_i$ from 
\eqref{eqlGnwsqjxTr} will be replaced by 1. 
For $i\in\set{3,4,\dots,v}$, 
\begin{align*}
&\sba d{m_1-1}=\sba{4v+1}{8v}\overset{\eqref{eqSprNrskzPnnmDlrCjsX}}{\geq}{{8v-1}\choose{4v}}\geq 8v-1 \geq 4v,\cr
&\sba {m_2-d}{m_2-d-1}=\sba {4v+2}{4v+1} \overset{\eqref{eqSprNrskzPnnmDlrCjsX}}{\geq}  {{4v}\choose{2v}}\geq  4v,\cr
&\sba {d+3-i}{m_i-d-1}=\sba {4v+4-i}{4v+2i-3}\cr
&\kern 3.5cm \overset{\eqref{eqSprNrskzPnnmDlrCjsX}}{\geq} 
{{4v+2i-4}\choose{2v+i-2}}\geq  
4v+2i-4 \geq    4v.
\end{align*}
Hence, \eqref{eqthmmaina}, \eqref{eqthmmainb}, and \eqref{eqthmmainc} together with Remark~\ref{remarkcZgBnjT} allow that
$p_1=q_1=p_2=\dots=q_v :=2v=u$.  Then $n_1', n_1'', n_2', n_2'',\dots, n_v', n_v''$ are $u=2v$ consecutive numbers and the corollary follows from Theorem~\ref{thmmain}.
\end{proof}

\section{Facts and  statistical analysis achieved with the help computer programs}\label{sectstatcomp}

\subsection{Estimating confidence intervals}\label{subsectconfint}
In this section, we are going to use some well-known facts of statistics; see, for example,  Hodges and Lehmann~\cite[page 255]{hodgeslehmann}, 
Lefebvre~\cite[Chapter 6.2]{lefebvre}, and mainly  Mendenhall, Beaver and Beaver~\cite[around page 60]{mendenhallatal}. For a beginner in statistics, like the first author, it is not so easy to extract the necessary tools from the literature with little effort. Hence, 
we recall and summarize what we need from statistics.
Note that an event of a binomial model corresponds to a random variable with two possible values, 0 and 1; this allows us to simplify what follows below.  
Note also that for a large $k$, $\sqrt{k/(k-1)}$ is very close to 1; for example, it is $1.000020000$ (up to nine digits) for $k=25000$. Hence, if we replaced $k-1$ by $k$ in \eqref{subsectconfint} as some sources of information do, then the error would be neglectable (and smaller than what rounding can cause).

Assume that an experiment has only two possible outcomes: ``success'' with probability $p$ and ``failure'' with probability $q:=1-p$ but none of $p$ and $q$ is known. In order to obtain some information on $p$, we take a random \emph{sample}, that is, we repeat the experiment $k$ times independently. 
Let $s$ denote the number of those  experiments that ended up with ``success''. Then, of course, 
\begin{equation}
\text{we estimate $p$ by $\ovl p:=s/k$},
\label{eqtxtvRlphszLk}
\end{equation}
but we also would like to know how much we can rely on this estimation. Therefore, 
we let $\ovl q:=1-\ovl p$, pick a ``confidence level'' $\conf\in (0,1)\subset \mathbb R$, we let
\begin{equation}
\ovl\sigma:=\sqrt{\frac {\ovl p \cdot \ovl q}{k-1}},
\end{equation}
and we determine the positive real number $\zconf$ from the equation   
\begin{equation}
\conf = \int_{-\zconf}^{\zconf} \frac 1{\sqrt{2\pi}}\cdot e^{-x^2/2} d x.
\label{eqnwsHjwz8KF}
\end{equation}
Note that the function to be integrated in \eqref{eqnwsHjwz8KF} is the so-called density function of the \emph{standard normal distribution} and the $\zconf$ for many typical values of $\conf$ are given in practically all books on statistics. In this paper, to maintain five-digit accuracy, we
are going to use the values given in Table~\ref{tabledkzdskTG}. Finally, we define the
\begin{equation}
 \text{\emph{confidence interval} $I(\conf)$ to be 
$[\,\ovl p-\zconf\ovl\sigma,\, \ovl p+\zconf\ovl\sigma\,]$.}
\label{eqcTfjWvRhgBnNk}
\end{equation}
Let us emphasize that while $p$ is a concrete real number, the confidence interval is \emph{random}, because it depends on a randomly chosen sample. Even if we take another $k$-element sample with the same $k$ (that is, we repeat the experiments $k$ times again), then (with very high probability in general) we obtain a different confidence interval. We cannot claim that the confidence interval $I(\conf)$ surely contains the unknown probability $p$. However, statistics says that 
\begin{equation}
\text{with probability at least $\conf$, \,\,$I(\conf)$ contains $p$.}
\label{eqtxtPrhBkszWmXx}
\end{equation}

\begin{table}
\[
\vbox{\tabskip=0pt\offinterlineskip
\halign{\strut#&\vrule#\tabskip=4pt plus 2pt&
#\hfill& \vrule\vrule\vrule#&
\hfill#&\vrule#&
\hfill#&\vrule#&
\hfill#&\vrule#&
\hfill#&\vrule\tabskip=0.1pt#&
#\hfill\vrule\vrule\cr
\vonal\vonal\vonal\vonal
&&\hfill$\conf$&&$0.900$&&$0.950$&&$0.990$&&$0.999$&
\cr\vonal\vonal
&&\hfill$\zconf$&&$1.64485$&&$1.95996$&&$2.57583$&&$3.29053$&
\cr\vonal
\vonal\vonal\vonal
}} 
\]
\caption{$\zconf$ for some confidence levels $\conf$; taken from 
\texttt{https://mathworld.wolfram.com/ConfidenceInterval.html}}\label{tabledkzdskTG}
\end{table}

Next, assume that $G$ is a given subset of a large finite set $F$. (For example, and this is what is going to happen soon, $F$ can be the set of all four-element subsets of $\Part n$, for $n\in\set{4,5,\dots,9}$, and $G$ can be $\set{H\in G: \sublat H=\Part n}$.)
We know the size $|F|$ of $F$ but that of $G$ is usually unknown for us. We would like to obtain some information on $|G|$; this task is equivalent to getting information on the portion $p:=|G| / |F|$. With respect to uniform distribution, $p$ is the probability that a randomly selected element of $F$ belongs  $G$. Take a $k$-element sample, that is, select $k$ members of $F$ independently,
and declare ``success'' if a randomly chosen member belongs to $G$. With a fixed confidence level $\conf$, compute the confidence interval $I(\zconf)$ according to \eqref{eqnwsHjwz8KF} and \eqref{eqcTfjWvRhgBnNk}. Then we conclude from \eqref{eqcTfjWvRhgBnNk} and \eqref{eqtxtPrhBkszWmXx} that
\begin{equation}\left.
\parbox{10cm}{with probability at least $\conf$, the $k$-element sample has been chosen so that
$(\ovl p-\zconf\ovl\sigma)\cdot|F|\leq |G|
\leq (\ovl p+\zconf\ovl\sigma)\cdot|F|$.
}\,\,\,\right\}
\label{eqpbxZhGrmWdkRsRh}
\end{equation}

\subsection{Computer programs}
The authors, working with computers independently, have developed two disjoint sets of computer programs; all data to be reported in (this) Section~\ref{sectstatcomp} were achieved by these programs. 
Furthermore, a sufficient amount of these data, including $\gnu 4=50$ and $\gnu 5=5\,305$ from Table~\ref{tableexact}, were achieved independently by both authors, with different programs, different attitudes to computer programming, and different computers. This fact gives us a lot of confidence in our programs and the results obtained by them even if some results that needed too much performance from our computers and programs were achieved only by one of the above-mentioned two settings. 

The first setting includes some programs written in 
Bloodshed \emph{Dev-Pascal} v1.9.2 (Freepascal) under Windows 10 and also in \emph{Maple} V. Release 5 (1997); these programs are available from the website of the first author; see the list of publications there.  The strategy is to represent a partition by a lexicographically ordered list of its blocks separated by zeros. For example, for $n=8$, the partition
$\set{\set{4,6},\set{1,5,3,7}, \set{2,8}}$ is represented by
the vector 
\begin{equation}
\tuple{1,3,5,7,0,2,8,0,4,6,0,-1,-1,-1,-1,-1,-1}
\label{eqnbhtWfkJSvBbK}
\end{equation}
where $-1$ means that there is no more block. This representation is unique, which allows us to order any set of partitions lexicographically. The benefit of this lexicographic ordering is that it takes only $O(t)$ step to decide if a given partition belongs to a $t$-element set of partitions; this task occurs many times when computing sublattices generated by four partitions. The collection of all partitions of $\set{1,2,\dots,n}$ was computed recursively by a Maple program. This collection was saved into a txt file. After inputting this file, the hard job of generating sublattices was done by Pascal programs, which took care of efficiency in some ways.
Since all the partitions were input by these Pascal programs into an array, a random partition was selected by selecting its index as a random number of the given range.
Note that Maple was also used to do some computations, trivial for Maple, to obtain some numbers occurring in this section; see, for example, the numbers in Table~\ref{tablehnyklNc}.

The second set  consist of programs 
 written in \emph{R} (64 bit, version 3.6.3)  and \emph{Python} 3.8 which works well with any operating system.
Implementing the urn model of Stam~\cite{stam} in R, random members of $\Part n$ were selected in a sophisticated way as follows; note that this method does not require that $\Part n$ be stored in the computer. The method consists of two steps. First, choose a positive integer $u$ 
according to the probability distribution  
\begin{equation}
P(u=j)= 
\frac{j^n}{e\cdot j!\cdot \Bell n}, \qquad j\in\NN,
\label{eqdBsnskHGrL}
\end{equation}
where $e$ is the well-known constant $\lim_{j\to\infty} (1+1/j)^j$.
It is pointed out in Stam~\cite{stam} that $\sum_{j=1}^\infty P(u=j)=1$, so \eqref{eqdBsnskHGrL} is 
indeed a probability distribution. Fortunately, the series in \eqref{eqdBsnskHGrL} converges very fast and  $\sum_{j=1}^{2n} P(u=j)$ is very close to 1. Hence, though the program chose $u$ from $\set{1,2,\dots, 2n}$ rather than from $\NN$, the error is neglectable. In the next step, put the numbers  1, 2, \dots, $n$ into $u$ urns at random, according to the uniform distribution on the set of urns. 
Finally, the contents of the nonempty urns constitute a random partition that we were looking for.  
The implementation of partitions involved importing \emph{sympy} (external library with functions for computing \emph{Partitions}) into Python. Additionally, \emph{itertool} which is a �built-in� function in Python was necessary for computing various combinations for the four partitions of which meets and joins were evaluated. The combination of these two functions was crucial for the computation of the sublattices generated by the four partitions.

The lion's share of the computation was done on a desktop computer
with AMD Ryzen 7 2700X Eight-Core Processor 3.70 GHz. With the speed 3.70 GHz, the total amount of ``pure computation time'' was a bit more than two and a half weeks; see Tables~\ref{tableexact} and \ref{tablestat}. By ``pure computation time'' we mean that a single copy of the program was running without being disturbed by other programs and without letting the computer go to an idle state. 
The whole computation took more than a month because of several breaks when the computer was idle or it was turned off or it was used for  other purposes.

\subsection{Data obtained by computer programs}
The results obtained by computers are given in Tables~\ref{tableexact} and \ref{tablestat}. In particular, Table~\ref{tableexact} gives the number $\gnu n$ of four-element generating sets for $n\in\set{4,5,6}$; see also \eqref{pbxgnujl} where the notation $\gnu n$ is defined. Clearly, $\gnu 7$ cannot be determined by our programs and computers, although this task might be possible with thousands or millions of similar computers working jointly for a few years or so. (But this is just a first impression not supported by real analysis.)

Table~\ref{tablestat} shows what we have obtained from random samples. For a given $n$, let 
\begin{equation}
\grho=\grho (n):= \gnu n \cdot {\Bell n\choose 4 }^{-1};
\label{eqzrhNjZjJzs}
\end{equation}
this is the exact theoretical probability that a random four-element subset of $\Part n$ generates $\Part n$. 
In accordance with \eqref{eqtxtvRlphszLk}, $\ogrho=\ogrho (n)$ is $s/k$; the table contains $100\ogrho$ up to five digits.  The least and the largest endpoints of the confidence interval $I(\conf)$ are denoted by ${\conf}_\ast$ and $\conf{}^\ast$, respectively. For example, with this notation,  $[0.999_\ast,\, 0.999^\ast]$ is the confidence interval $I(0.999)$. 

In order to enlighten the meaning of Table~\ref{tablestat} even more, consider, say,  the  entries of its last two 
rows in the column for $n=7$. Taking \eqref{eqpbxZhGrmWdkRsRh} and Table~\ref{tablehnyklNc} also into account, we obtain that 
\begin{align}
&[0.0157753,\,\, 0.0159877]\text{ contains }\grho (7) 
\text{ with probability }0.999\text{ and}
\label{eqtxtdzRmfMhBgrnGr}
\\
&[3.86180\cdot 10^8,\,\, 3.91381 \cdot 10^8]\text{ contains }\gnu 7 
\text{ with probability }0.999.
\label{eqtxlGhermtsb5zNt}
\end{align}
However, let us note the following. From rigorous mathematical point of view, not even $\gnu 7\geq 2$ is proved by  Table~\ref{tablestat}. Indeed, it is theoretically possible (although very unlikely) that all the $238\,223$ generating sets the program found after fifteen million experiments are the same. 
The right interpretation of \eqref{eqtxtdzRmfMhBgrnGr} and \eqref{eqtxlGhermtsb5zNt} is that if very many lattice theorists and statisticians pick  random samples of the same size (that is, they pick fifteen million four-elements subsets of $\Part n$ and count the generating sets among them), then these samples give many different intervals but approximately 99.9 percent of these colleagues find intervals that happen to contain $\grho(7)$ and  $\gnu 7$, respectively. We the authors hope that we belong to these 99.9 percent.

When only Proposition~\ref{proplwStmtsH}, Table~\ref{tableexact} and so $\gnu 4$, $\gnu 5$,  and $\gnu 6$ were only known, it was not clear whether $\lim_{n\to\infty} \grho(n)$ is zero or not. But now, based on  experience with generating sets of $\Equ n$ and Table~\ref{tablestat}, we risk formulating the following conjecture.

\begin{conjecture} The set $\set{\grho(n): 4\leq n\in\NN}$ has a positive lower bound.
\end{conjecture} 

It would be too early to formulate the rest of our feelings as a conjecture, so we formulate it as an open problem. 
\begin{problem} Is it true that  
\[
\grho(6)>\grho(7)\leq \grho(8)\leq \grho(9)\leq \grho(10)\leq \grho(11)\leq\dots
\] 
and, for all $4\leq n\in \NN$, $\grho(n)\geq 3/200$? Note that we already know from Table~\ref{tableexact}  that $\grho(4)> \grho(5)> \grho(6)>3/200$.
\end{problem}

Since $\grho(6)$  is exactly known, our confidence in $\grho(6)>\grho(7)$ at least $0.999$. However, we are not really confident in, say, $\grho(8)\leq \grho(9)$ even if this inequality is more likely to hold than to fail.

\bigskip

\begin{table}
\[
\vbox{\tabskip=0pt\offinterlineskip
\halign{\strut#&\vrule#\tabskip=4pt plus 2pt&
#\hfill& \vrule\vrule\vrule#&
\hfill#&\vrule#&
\hfill#&\vrule#&
\hfill#&\vrule\tabskip=0.1pt#&
#\hfill\vrule\vrule\cr
\vonal\vonal\vonal\vonal
&&\hfill$n$&&$\,7$&&$\,8$&&$\,9$&
\cr\vonal\vonal
&&\hfill$\displaystyle{{\Bell n}\choose 4}$&&$24\,480\,029\,875$&&$12\,222\,513\,708\,615$&&$8\,330\,299\,023\,110\,190$&
\cr\vonal
\vonal\vonal\vonal
}} 
\]
\caption{The number of four-element subsets of $\Part n$ for $n\in\set{7,8,9}$}\label{tablehnyklNc}
\end{table}

\begin{table}
\[
\vbox{\tabskip=0pt\offinterlineskip
\halign{\strut#&\vrule#\tabskip=4pt plus 2pt&
#\hfill& \vrule\vrule\vrule#&
\hfill#&\vrule#&
\hfill#&\vrule#&
\hfill#&\vrule\tabskip=0.1pt#&
#\hfill\vrule\vrule\cr
\vonal\vonal\vonal\vonal
&&\hfill$n$&&$\,4$&&$\,5$&&$\,6$&
\cr\vonal\vonal
&&\hfill$\displaystyle{{\Bell n}\choose 4}$&&$1\,365$&&$270\,725$&&$68\,685\,050$&
\cr\vonal
&&$\hfill\gnu n$&&$50$&&$5\,305$&&$1\,107\,900$&
\cr\vonal
&&\hfill \%, i.e., $100\grho(n)$&&$3.663003663$&&$1.959553052$&&$1.613014768$&
\cr\vonal
&&\hfill computer time&&$0.11$sec&&$68$ sec&&$38$ hours&
\cr\vonal
\vonal\vonal\vonal
}} 
\]
\caption{The (exact) number $\gnu n$ of the four-element generating sets of $\Equ n$ for $n\in\set{4,5,6}$}\label{tableexact}
\end{table}

\begin{table}
\[
\vbox{\tabskip=0pt\offinterlineskip
\halign{\strut#&\vrule#\tabskip=2pt plus 2pt&
#\hfill& \vrule\vrule\vrule#&
\hfill#&\vrule#&
\hfill#&\vrule#&
\hfill#&\vrule#&
\hfill#&\vrule#&
\hfill#&\vrule#&
\hfill#&\vrule\tabskip=0.1pt#&
#\hfill\vrule\vrule\cr
\vonal\vonal\vonal\vonal
&&\hfill$n$&&$\,4$&&$\,5$&&$\,6$&&$\,7$&&$8$&&$9$&
\cr\vonal
&&\hfill$100\grho(n) $&&$3.6630037$&&$1.9595531$&&$1.613014768$&&$\phantom x$&&$\phantom x$&&$\phantom x$&
\cr\vonal\vonal\vonal\vonal
&&\hfill$k$&&$10\,000\,000$&&$10\,000\,000$&&$10\,000\,000$&&$15\,000\,000$&&$500\,000$&&$25\,000$&
\cr\vonal
&&\hfill time&&$8$ minutes&&$27$ min&& \hfill$\text{3h+33min}$&&$102$ hours&&$95$ h&&$166$ h&
\cr\vonal
&&\hfill $s$&&$367\,221$&&$196\,243$&&$161\,768$&&$238\,223$&&$8\,244$&&$438$&
\cr\vonal
&&$\hfill 100\ogrho(n)$&&$3.67221$&&$1.96243$&&$1.61768$&&$1.58815$&&$1.64880$&&$1.75200$&
\cr\vonal\vonal\vonal\vonal
\isitneeded{
&&$0.800_\ast$&&$3.66459$&&$1.95681$&&$1.61257$&&$1.58402$&&$1.62572$&&$1.64566$&
\cr\vonal
&&$0.800^\ast$&&$3.67983$&&$1.96805$&&$1.62279$&&$1.59229$&&$1.67188$&&$1.85834$&
\cr\vonal\vonal\vonal\vonal
}
&&$0.900_\ast$&&$3.66243$&&$1.95522$&&$1.61112$&&$1.58284$&&$1.61918$&&$1.61551$&
\cr\vonal
&&$0.900^\ast$&&$3.68199$&&$1.96964$&&$1.62424$&&$1.59346$&&$1.67842$&&$1.88849$&
\cr\vonal\vonal\vonal\vonal
&&$0.950_\ast$&&$3.66055$&&$1.95383$&&$1.60986$&&$1.58183$&&$1.61350$&&$1.58936$&
\cr\vonal
&&$0.950^\ast$&&$3.68387$&&$1.97103$&&$1.62550$&&$1.59448$&&$1.68410$&&$1.91464$&
\cr\vonal\vonal\vonal\vonal
&&$0.990_\ast$&&$3.65689$&&$1.95113$&&$1.60740$&&$1.57984$&&$1.60241$&&$1.53826$&
\cr\vonal
&&$0.990^\ast$&&$3.68753$&&$1.97373$&&$1.62796$&&$1.59647$&&$1.69519$&&$1.96574$&
\cr\vonal\vonal\vonal\vonal
&&$0.999_\ast$&&$3.65264$&&$1.94800$&&$1.60455$&&$1.57753$&&$1.58954$&&$1.47896$&
\cr\vonal
&&$0.999^\ast$&&$3.69178$&&$1.97686$&&$1.63081$&&$1.59877$&&$1.70806$&&$2.02504$&
\cr\vonal
\vonal\vonal\vonal
}} 
\]
\caption{Statistics with $k$ experiments that yielded $s$  many 4-element generating sets of $\Part n$ for $n\in\set{4,\dots,9}$}\label{tablestat}
\end{table}%

\end{document}